\documentclass[12pt]{amsart}
\usepackage{amsmath,amssymb,enumerate}
\usepackage[pdftex,pagebackref=false]{hyperref} 
\newtheorem{theorem}{Theorem}[section]
\newtheorem{claim}{}[theorem]
\newtheorem{lemma}[theorem]{Lemma}

\newtheorem{corollary}[theorem]{Corollary}
\newtheorem{conjecture}[theorem]{Conjecture}
\theoremstyle{definition}

\newcommand{\cR}{\mathcal{R}}
\newcommand{\cF}{\mathcal{F}}

\newcommand{\cM}{\mathcal{M}}

\newcommand{\cS}{\mathcal{S}}
\newcommand{\cT}{\mathcal{T}}
\newcommand{\cU}{\mathcal{U}}
\newcommand{\cX}{\mathcal{X}}
\newcommand{\cY}{\mathcal{Y}}
\newcommand{\cV}{\mathcal{V}}
\newcommand{\cW}{\mathcal{W}}
\newcommand{\cH}{\mathcal{H}}
\newcommand{\fX}{\mathfrak{X}}

\DeclareMathOperator{\cl}{cl}
\DeclareMathOperator{\PG}{PG}
\DeclareMathOperator{\GF}{GF}

\DeclareMathOperator{\wt}{wt}
\newcommand{\elem}{\epsilon}
\newcommand{\del}{\setminus}
\newcommand{\con}{/}
\newcommand{\dcon}{\con}
\newcommand{\oldcon}{\con}
\newcommand{\fU}{\cU}

\hypersetup{
    plainpages=false,       
    unicode=false,          
    pdffitwindow=false,     
    pdfstartview={FitH},    
    pdftitle={uWaterloo\ LaTeX\ Thesis\ Template},   
    pdfnewwindow=true,      
    colorlinks=true,        
    linkcolor=black,         
    citecolor=black,        
}

\title[Exponentially Dense Matroids]{Projective Geometries in Exponentially Dense Matroids. I}
\author[Geelen]{Jim Geelen}
\author[Nelson]{Peter Nelson}		
	\begin{document}	
	\maketitle
	\begin{abstract}
	We show for each positive integer $a$ that, if $\cM$ is a minor-closed class of matroids not 
	containing all rank-$(a+1)$ uniform matroids, then there exists an integer $n$ such that either 
	every rank-$r$ matroid in $\cM$ can be covered by at most $r^n$ sets of rank at most $a$, or $\cM$ contains the
	$\GF(q)$-representable matroids for some prime power $q$, and every rank-$r$ matroid in $\cM$ can be 
	covered by at most $r^nq^r$ sets of rank at most $a$. This determines the maximum density of the matroids in $\cM$ 
	up to a polynomial factor. 
	\end{abstract}	
			
	\section{Introduction}
	
	If $M$ is a matroid, and $a$ is a positive integer, then $\tau_a(M)$ denotes the 
	\emph{$a$-covering number} of $M$, the minimum number of sets of rank at most $a$ 
	in $M$ required to cover $E(M)$. We will prove the following theorem: 
	
	\begin{theorem}\label{mainnice}
		Let $a$ be a positive integer. If $\cM$ is a minor-closed class of matroids, then either
		\begin{enumerate}
			\item $\tau_a(M) \le r(M)^{n_{\cM}}$ for all $M \in \cM$, or
			\item\label{mii} there is a prime power $q$ so that $\tau_a(M) \le r(M)^{n_{\cM}}q^{r(M)}$ 
			for all $M \in \cM$, 
				and $\cM$ contains all $\GF(q)$-representable matroids, or
			\item\label{miii} $\cM$ contains all rank-$(a+1)$ uniform matroids. 
		\end{enumerate}
	\end{theorem}

	Here, $n_{\cM}$ denotes an integer constant depending only on $\cM$.  In [\ref{part2}], the second author will 
	refine the bound $r(M)^{n_{\cM}}q^{r(M)}$ in (\ref{mii}) by a polynomial factor to $c_{\cM} q^{r(M)}$ for some 
	constant $c_{\cM}$; it is routine to show that this improved bound is best-possible up to a constant factor. 
	Both these results also appear in [\ref{thesis}]. 
	
	The above theorem and its improvement in [\ref{part2}] are contained in the following larger conjecture of 
	the first author [\ref{openprobs}]:
	
	\begin{conjecture}[Growth Rate Conjecture]\label{grc}
		Let $a \ge 1$ be an integer. If $\cM$ is a minor-closed class of matroids, then either
		\begin{enumerate}
			\item $\tau_a(M) \le c_{\cM}r(M)$ for all $M \in \cM$, or
			\item $\tau_a(M) \le c_{\cM}r(M)^2$ for all $M \in \cM$, and $\cM$ contains all graphic matroids or
				all bicircular matroids, or
			\item there is a prime power $q$ such that $\tau_a(M) \le c_{\cM}q^{r(M)}$ for all $M \in \cM$ 
				and $\cM$ contains all $\GF(q)$-representable matroids, or
			\item\label{mciv} $\cM$ contains all rank-$(a+1)$ uniform matroids. 
		\end{enumerate}
	\end{conjecture}
	
	When $a = 1$, the parameter $\tau_a(M)$ is just the number of points in $M$, sometimes
	written as $\elem(M)$, and (\ref{mciv}) corresponds to $\cM$ containing all simple rank-$2$ 
	matroids. The conjecture in this case was proved by 
	work of Geelen, Kabell, Kung and Whittle [\ref{gk},\ref{gkw},\ref{gw}], and stated in [\ref{gkw}]
	as the `Growth Rate Theorem'.
	
	For general $a$, if (\ref{mciv}) holds then there is no bound on $\tau_a(M)$ as a function of 
	$r(M)$ for all $M \in \cM$, as a rank-$(a+1)$ uniform matroid can require arbitrarily many 
	rank-$a$ sets to cover. Thus, we derive bounds on $\tau_a$ relative to some particular 
	rank-$(a+1)$ uniform matroid that is excluded as a minor. We prove Theorem~\ref{mainnice} 
	as a consequence of the following result:
	
	\begin{theorem}\label{halfwaypoint}
		For all integers $1 \le a < b$, $q \ge 1$ and $n \ge 1$, there exists
		an integer $m$ so that, if $M$ is a matroid
		of rank at least $2$ with no $U_{a+1,b}$-minor, 
		and $\tau_a(M) \ge r(M)^{m}q^{r(M)}$, 
		then $M$ has a $\PG(n-1,q')$-minor for some prime power $q' > q$. 
	\end{theorem}	
	
	Our proof is loosely based on ideas in [\ref{gk}], and uses its main results as a base case. 
	The next five sections are used to define the terminology and intermediate structures we need, 
	and the bulk of the argument rests on the lemmas in Sections~\ref{u1s5},~\ref{u1s6} and~\ref{u1s7}.

\section{Preliminaries}

	We follow the notation of Oxley [\ref{oxley}]. Two sets $X$ and $Y$ are \emph{skew} in a matroid $M$ 
	if $r_M(X \cup Y) = r_M(X) + r_M(Y)$, and a collection of sets $\cX$ in $M$ is \emph{mutually skew} if 
	$r_M\left(\cup_{X \in \cX} X \right) = \sum_{X \in \cX} r_M(X)$. Very often, the atomic objects in our 
	proof are sets in $M$ rather than elements; to this end, we also define some new notation. 
	
	A common object is a collection of sets of the same rank. If $M$ is a matroid, and $a \ge 1$ is an integer, then
	$\cR_a(M)$ denotes the set $\{X \subseteq E(M): r_M(X) = a\}$. 
	
	Generalising the notion of parallel elements, if $X,X' \subseteq E(M)$, then we write $X \equiv_M X'$ if 
	$\cl_M(X) = \cl_M(X')$; we say that $X$ and $X'$ are \textit{similar} in $M$. We
	write $[X]_M = \{X' \subseteq E(M): X \equiv_M X'\}$ for the `similarity class' of $X$ in $M$. 
	
	We also extend existing notation in straightforward ways. If $\cX \subseteq 2^{E(M)}$ is a collection of sets, 
	then we write 	$M|\cX$ for $M|(\cup_{X \in \cX} X)$, $\cl_M(\cX)$ for $\cl_M(\cup_{X \in \cX} X)$, 
	and $r_M(\cX)$ for 
	$r(M|\cX)$. Two sets $\cX, \cX' \subseteq 2^{E(M)}$ are \textit{similar} in $M$ if $\cl_M(\cX) =\cl_M(\cX')$. 
	
	Analogously to the notion of a simple matroid, we say that $\cX \subseteq 2^{E(M)}$ is \emph{simple} in $M$ if the
	sets in $\cX$ are pairwise dissimilar in $M$. Note that any collection of flats of $M$ is simple. We write
	$\elem_M(\cX)$ for the maximum size of a subset of $\cX$ that is simple in $M$, or
	equivalently the number of different similarity classes of $2^{E(M)}$ containing a set in $\cX$. If $\cX$ 
	just contains nonloop singletons, then $\elem_M(\cX) = \tau_1(M|\cX)$.

	For integers $a$ and $b$ with $1 \le a < b$, we write $\cU(a,b)$ for the class of matroids 
	with no $U_{a+1,b}$-minor. 
	The first tool in our proof is a theorem of Geelen and Kabell [\ref{gkb}], which shows that the 
	parameter $\tau_a$ is bounded as a function of rank across $\cU(a,b)$. 
	
	\begin{theorem}\label{kdensity}
			Let $a$ and $b$ be integers with $1 \le a < b$. If $M \in \cU(a,b)$ satisfies $r(M) > a$,
			then $\tau_a(M) \le \binom{b-1}{a}^{r(M)-a}$. 
		\end{theorem}
		\begin{proof}
			We first prove the result when $r(M) = a+1$, then proceed by induction. If $r(M) = a+1$, observe that 
			$M|B \cong U_{a+1,a+1}$ for any basis $B$ of $M$; let $X \subseteq E(M)$ be maximal such that
			$M|X \cong U_{a+1,|X|}$. We may assume that $|X| < b$, and by maximality of $X$ every $e \in E(M)-X$ 
			is spanned by a rank-$a$ set of $X$. Therefore, $\tau_a(M) \le \binom{|X|}{a} \le \binom{b-1}{a}$. 
			
			Suppose that $r(M) > a+1$, and inductively assume that the result holds for matroids of smaller rank. 
			Let $e$ be a nonloop of $M$. We have $\tau_{a+1}(M) \le \tau_a(M \con e) \le \binom{b-1}{a}^{r(M)-a-1}$ by induction, 
			and by the base case, each rank-$(a+1)$ set in $M$ admits a cover with at most $\binom{b-1}{a}$ sets of 
			rank at most $a$. Therefore $\tau_a(M) \le \binom{b-1}{a} \tau_{a+1}(M) \le \binom{b-1}{a}^{r(M)-a}$, 
			as required. 
		\end{proof}

	This theorem has two simple corollaries concerning the density of matroids in $\cU(a,b)$
	relative to that of their minors:
	\begin{lemma}\label{kdensitycon}
		Let $a$ and $b$ be integers with $1 \le a < b$. If $M \in \cU(a,b)$ and $C \subseteq E(M)$, then
		$\tau_a(M) \le \binom{b-1}{a}^{r_M(C)}\tau_a(M \con C)$.  
	\end{lemma}
	
	\begin{lemma}\label{kdensityrel}
		Let $a$ and $b$ be integers with $1 \le a < b$. If $M \in \cU(a,b)$ and $N$ is a minor of $M$,
		then $\tau_a(M) \le \binom{b-1}{a}^{r(M)-r(N)}\tau_a(N)$. 
	\end{lemma}
	
	The next two theorems, also due to Geelen and Kabell, were proved in [\ref{gk}] to resolve the 
	`polynomial-exponential' part of the Growth Rate Theorem, both finding a large projective geometry 
	in a sufficiently dense matroid without some line as a minor:
	
	\begin{theorem}\label{gkpoly}
		There is an integer-valued function $f_{\ref{gkpoly}}(\ell,n)$ such that, for any integers $\ell \ge 2$ and
		$n \ge 2$, if $M \in \cU(1,\ell)$ satisfies $\tau_1(M) \ge r(M)^{f_{\ref{gkpoly}}(\ell,n)}$,
		then $M$ has a rank-$n$ projective geometry minor.
	\end{theorem}
	
	\begin{theorem}\label{gkexp}
		There is a real-valued function $\alpha_{\ref{gkexp}}(\ell,n,q)$ so that, for any integers $q \ge 2$,
		$\ell \ge 2$ and $n \ge 1$, if $M \in \cU(1,\ell)$ satisfies $\tau_1(M) \ge \alpha_{\ref{gkexp}}
		(\ell,n,q) q^{r(M)}$, then $M$ has a $\PG(n-1,q')$-minor for some prime power $q' > q$. 
	\end{theorem}	
	
\section{Thickness and Firmness}\label{u1s1}
	
	Two density-related notions that will feature frequently in our proof are those of \textit{thickness} and
	\textit{firmness}, which we define and explain in this section.
	
	If $d$ is a positive integer and $M$ is a matroid, then $M$ is \textit{$d$-thick} if 
	$\tau_{r(M)-1}(M) \ge d$. A set $X \subseteq E(M)$ is \textit{$d$-thick in $M$} if $M|X$ is $d$-thick. 
	
	Note that every matroid is $2$-thick, and that thickness is monotone in the sense that if $d' \ge d$ and $M$ is 
	$d'$-thick, then $M$ is $d$-thick. The following lemma is fundamental and we use it freely and frequently in our
	proof. 
	
	\begin{lemma}\label{maintainthickness}
		Let $d \ge 1$ be an integer. If $M$ is a matroid, $N$ is a minor of $M$, and $X \subseteq E(N)$ is 
		$d$-thick in $M$, then $X$ is $d$-thick in $N$. 
	\end{lemma}
	\begin{proof}
		Deleting an element of $M$ outside $X$, or contracting an element outside $\cl_M(X)$ does not change $M|X$, 
		so it suffices to show that contracting a nonloop $e \in \cl_M(X)$ does not destroy $d$-thickness of $X$. 
		This follows from the fact that $\tau_{r(M)-2}(M \con e) \ge \tau_{r(M)-1}(M)$.
	\end{proof}
	
	Any rank-$1$ or rank-$0$ matroid is clearly arbitrarily thick. Convenient examples of thick matroids are uniform
	matroids - no rank-$a$ set in the matroid $U_{a+1,b}$ contains more than $a$ elements, so $U_{a+1,b}$ is 
	$\lceil \frac{b}{a} \rceil$-thick. Indeed, sufficient thickness and rank ensure a large uniform minor:
	
	\begin{lemma}\label{thickuniform}
		Let $a$ and $b$ be integers with $1 \le a < b$. If $M$ is $\binom{b}{a}$-thick and $r(M) > a$,
		then $M$ has a $U_{a+1,b}$-minor. 
	\end{lemma}
	\begin{proof}
		By Lemma~\ref{maintainthickness}, $d$-thickness of $M$ is preserved by contraction, so by 
		contracting points if needed, we may assume that $r(M) = a+1$. Now, 
		$\binom{b-1}{a} < \binom{b}{a} \le \tau_a(M)$, so the result follows from 
		Theorem~\ref{kdensity}.
	\end{proof}
	
	This lemma tells us that, qualitatively, searching for a $U_{a+1,b}$-minor is equivalent to searching 
	for an appropriately thick minor of rank greater than $a$. We take this approach hereon; in fact, nearly 
	all the uniform minors we find will be constructed by implicit use of this lemma. 
	
	We now turn to a definition of firmness. If $d \ge 1$ is an integer and $M$ is a matroid, then a set
	$\cX \subseteq 2^{E(M)}$ is \textit{$d$-firm} in $M$ if all $\cX' \subseteq \cX$ with 
	$|\cX'| > d^{-1}|\cX|$ satisfy $r_M(\cX') = r_M(\cX)$. 

	Firmness is a measure of how `evenly spread' a collection of sets is. The set of points in a
	$d$-point line is $d$-firm; more generally, the set of $a$-subsets of $E(U_{a+1,b})$ is 
	$\binom{b}{a}$-firm. Firmness is clearly monotone in the sense that $d$-firmness implies 
	$(d-1)$-firmness. 
	
	Our first lemma precisely relates firmness to thickness:
	
	\begin{lemma}\label{thicknessfirmness}
		Let $a \ge 1$ and $d \ge 1$ be integers, and $M$ be a matroid. 
		If $\cX \subseteq \cR_a(M)$ is $d$-firm in $M$, and each $X \in \cX$ is $d$-thick in $M$, 
		then $\cl_M(\cX)$ is $d$-thick in $M$.
	\end{lemma}
	
	\begin{proof}
		Let $\cF$ be a cover of $M|\cl_M(\cX)$ with flats of smaller rank; we wish to show that 
		$|\cF| \ge d$. If a set $X \in \cX$ is not contained in any flats in $\cF$, then 
		$\{X \cap F: F \in \cF\}$ is a cover of $M|X$ with sets of smaller rank, of size at
		most $|\cF|$, so $|\cF| \ge d$ by $d$-thickness of $X$. We may therefore assume 
		every $X \in \cX$ is contained in some $F \in \cF$. Now, since $\cX$ is $d$-firm in 
		$M$ and no flat in $\cF$ is spanning in $M|\cl_M(\cX)$, each flat in $\cF$ contains 
		at most $d^{-1}|\cX|$ different sets in $\cX$. We thus have 
		$|\cF| \ge \frac{|\cX|}{d^{-1}|\cX|} = d$, as required. 
	\end{proof}

	We will use this lemma to construct the thick sets of rank greater than $a$ that we are frequently
	seeking. Thus, we often consider a set $\cX \subseteq \cR_a(M)$ that has no firm subset of rank exceeding $a$ 
	in a minor of $M$; we are `excluding' a minor with this structure from $\cX$ and $M$ in lieu 
	of excluding $U_{a+1,b}$. 
	
	This exclusion allows us to control the number of sets in $\cX$ in useful ways; the first of the
	next two lemmas tells us about the `absolute' density of $\cX$ in $M$, and the second about the
	`relative' density of $\cX$ in $M$ as compared to in a minor of $M$. 
	
	\begin{lemma}\label{firmdensity}
		Let $a \ge 1$ and $d \ge 2$ be integers, $M$ be a matroid with $r(M) > a$, 
		and $\cX \subseteq \cR_a(M)$. If $\elem_M(\cX) \ge d^{r(M)-a}$, then there 
		is a set $\cY \subseteq \cX$ such that $r_M(\cY) > a$, and $\cY$ is $d$-firm in $M$. 
	\end{lemma}
	\begin{proof}
		We may assume that $\cX$ is simple. If $r(M) = a+1$, then the union of any two sets 
		in $\cX$ is spanning in $M$, and $|\cX| \ge d$, so $\cX$ is $d$-firm; we assume that 
		$r(M) > a+1$, and proceed by induction on $r(M)$.  If $\cX$ is not $d$-firm, then 
		there is some $\cX' \subseteq \cX$ with $r_M(\cX') < r_M(\cX)$, and 
		$|\cX'| \ge d^{-1}|\cX| \ge d^{r(M)-1-a} \ge d^{r_M(\cX')-a}$. Moreover, 
		$|\cX'| \ge d^{r(M)-1-a} \ge d \ge 2$, so $r_M(\cX') > a$. The result follows
		by applying the inductive hypothesis to $\cX'$ in $M|\cl_M(\cX')$. 
	\end{proof}
	
	\begin{lemma}\label{firmdensityrel}
		Let $a \ge 1$ and $d \ge 2$ be integers, $M$ be a matroid, $N$ be a minor of $M$, 
		and $\cX \subseteq \cR_a(M) \cap \cR_a(N)$. If $\elem_M(\cX) > d^{r(M)-r(N)}\elem_N(\cX)$, 
		then there is a set $\cY \subseteq \cX$ such that $r_M(\cY) > a$, and $\cY$ is $d$-firm in $M$. 
	\end{lemma}
	\begin{proof}
		Let $N = M \con C \del D$, where $r_M(C) = r(M)-r(N)$. Suppose that 
		$\elem_M(\cX) > d^{r_M(C)}\elem_N(\cX)$. By a majority argument applied 
		to the similarity classes of $\cX$ in $N$, there is some $X \in \cX$ such 
		that $\elem_M([X]_N \cap \cX) \ge d^{r_M(C)} = d^{r_M(X \cup C)-a}$. Now, 
		every set in $[X]_N \cap \cX$ is contained in $\cl_M(X \cup C)$, so 
		applying Lemma~\ref{firmdensity} to $M|(\cl_M(X \cup C))$ gives the result.
	\end{proof}
	
	\section{Arrangements}\label{u1s2}
	
	We prove two lemmas related to how collections of sets in a matroid `fit together'. 
	This first lemma shows that, given $\cX \subseteq \cR_a(M)$, we can contract a 
	point of $M$ so that the rank of most sets in $\cX$ is unchanged:
	
	\begin{lemma}\label{pickcontract}
		Let $M$ be a matroid of rank at least $1$, $a \ge 1$ be an integer, and $\cX \subseteq \cR_a(M)$. 
		There exists a nonloop $e \in E(M)$ so that 
		\[\elem_{M}(\cX \cap \cR_a(M \con e)) \ge \left(1 - \tfrac{a}{r(M)}\right)\elem_M(\cX).\]  
	\end{lemma}
	\begin{proof}
		Let $\cX'$ be a maximal simple subset of $\cX$, and $B$ be a basis of $M$. 
		Each set in $\cX'$ has at most $a$ elements of $B$ in its closure, so 
		$\sum_{f \in B} |\{X \in \cX': f \in \cl_M(X)\}| \le a|\cX'|$. There is 
		therefore some $e \in B$ such that $|\{X \in \cX': e \in \cl_M(X)\}| \le \frac{a}{|B|}|\cX'|$. 
		Every set in $\cX'$ that does not span $e$ is in $\cR_a(M \con e)$, so 
		\begin{align*}
			\elem_M(\cX \cap \cR_a(M \con e)) &\ge |\cX' \cap \cR_a(M \con e)|\\ 
			&\ge |\cX'| - \tfrac{a}{|B|}|\cX'|\\ &= \elem_M(\cX) - \tfrac{a}{r(M)} \elem_M(\cX),
		\end{align*}
		and the result follows.
	\end{proof}

	This second lemma relates to the fact that a graph with many edges contains 
	either a vertex of large degree or a large matching. Recall that $\cW \subseteq 2^{E(M)}$ 
	is mutually skew in $M$ if $r_M(\bigcup_{W \in \cW}W) = \sum_{W \in \cW} r_M(W)$.
	
	\begin{lemma}\label{findskew}
		Let $M$ be a matroid, $a \ge 1$ and $t \ge 1$ be integers, and let $\cX \subseteq \cR_a(M)$. Either
		\begin{enumerate}[(i)]
			\item there exists $\cW \subseteq \cX$ so that $|\cW| = t$, and $\cW$ is mutually skew in $M$, or
			\item there is a minor $N$ of $M$, a set $\cY \subseteq \cX \cap \cR_a(N)$, and a nonloop $e$ of $N$
				such that $r(N) \ge r(M)-at$,  and $|\cY| \ge (at)^{-1}|\cX|$, and $e \in \cl_{N}(Y)$ for all $Y \in \cY$. 
		\end{enumerate}
	\end{lemma}
	\begin{proof}
		Let $\cW$ be a maximal mutually skew subset of $\cX$; we may assume that 
		$k < t$. Let $e_1, \dotsc, e_{a|\cW|}$ be a basis for $\bigcup_{W \in \cW} W$.  
		For each $1 \le i \le a|\cW|$, let $M_i = M \con \{e_1, \dotsc, e_i\}$. 
		By maximality of $\cW$, each $X \in \cX-\cW$ satisfies $r_{M_{a|\cW|}}(X) < r_M(X) = a$, 
		and this inequality clearly also holds for all $X \in \cW$, so for each $X \in \cX$ 
		there is some $i_X$ such that $M|X = M_{i_X-1}|X$, and $X$ spans $e_{i_X}$ in $M_{i_X-1}$.
		By a majority argument, there is some $1 \le i_0 \le a|\cW|$ and $\cY \subseteq \cX$ 
		such that $|\cY| \ge (a|\cW|)^{-1}|\cX|$, and $i_Y = i_0$ for all $Y \in \cY$. 
		Since $|\cW| < t$, the minor $N = M \oldcon \{e_1, \dotsc, e_{i_0-1}\}$, 
		along with $\cY$ and $e_{i_0}$, will satisfy the second outcome. 
	\end{proof}
	
	\section{Weighted Covers and Scatteredness}\label{u1s4}
		
	Our main theorem concerns upper bounds on the parameter $\tau_a$. It is therefore 
	natural to consider minimum-sized covers of a matroid with sets of rank at most $a$. 
	However, such a cover has few useful properties. We will therefore change the parameter we are considering to one 
	that considers minimal `weighted' covers. This tweak will force a minimal cover to 
	have many properties that we exploit at length. 
	
	If $M$ is a matroid, and $\cX,\cF \subseteq 2^{E(M)}$, then $\cF$ is a 
	\textit{cover of $\cX$} in $M$ if every set in 
	$\cX$ is contained in a set in $\cF$. A \textit{cover of $M$}
	is a cover of $\{\{e\}: e \in E(M)\}$. 
	
	If $d$ is a positive integer and $\cF \subseteq 2^{E(M)}$, then we write 
	$\wt^d_M(\cF)$ for the sum $\sum_{F \in \cF}d^{r_M(F)}$, which we call
	the \textit{weight} of $\cF$. 
	Thus, the `weight' of a point in $\cF$ is $d$, the `weight' of a line is $d^2$, etc. 
	$\cF$ is a \textit{$d$-minimal cover of $\cX$} 
	if $\cF$ minimizes $\wt_M^d(\cF)$ subject to being a cover of $\cX$.
	We write $\tau^d(M)$ for the weight of a $d$-minimal cover of $M$. 
	The parameter $\tau^d$ will not drop too dramatically in a minor:
	
	\begin{lemma}\label{dcoverdensity}
		Let $d \ge 1$ be an integer. If $N$ is a minor of a matroid $M$, 
		then $\tau^d(N) \ge d^{r(N)-r(M)}\tau^d(M)$. 
	\end{lemma}
	\begin{proof}
		It suffices to show that, for a nonloop $e \in E(M)$, we have 
		$\tau^d(M \con e) \ge d^{-1}\tau^d(M)$. If $\cF$ is a $d$-minimal cover of $M \con e$, 
		then $\cF' = \{\cl_M(F \cup \{e\}): F \in \cF\}$ is a cover of $M$, so 
		$\tau^d(M) \le \wt^d_M(\cF') = \sum_{F \in \cF}d^{r_M(F \cup \{e\})} 
		= \sum_{F \in \cF} d^{r_{M \con e}(F)+1} = d \wt_{M \con e}^d(\cF) = d \tau^d(M \con e)$, 
		giving the result.
	\end{proof}
	
	A concept that we will soon use to build highly structured  minors is that of \textit{scatteredness}, 
	another measure of how `spread out' a collection of sets is. A set $\cX \subseteq 2^{E(M)}$ 
	is \textit{$d$-scattered} in a matroid $M$ if all sets in $\cX$ 
	are $d$-thick in $M$, and $\{\cl_M(X): X \in \cX\}$ is a $d$-minimal cover of $\cX$ in $M$. 

	A scattered set is a collection of thick sets that cannot be more efficiently covered with sets of larger rank. Again, we 
	use the symbol $d$; this same parameter will be passed around our proofs in measures of thickness, firmness and scatteredness. 
	
	Our first lemma establishes some nice properties in the case where a minimal cover of a 
	set $\cX \subseteq \cR_a(M)$ is just the ground set of $M$:
	
	\begin{lemma}\label{maximalminimalcover}
		Let $a \ge 1$ and $d \ge 1$ be integers, $M$ be a matroid with $r(M) > a$, and 
		$\cX \subseteq \cR_a(M)$. If all sets in $\cX$ are $d$-thick in $M$, and $\{E(M)\}$ 
		is a $d$-minimal cover of $\cX$ in $M$, then  $\elem_M(\cX) \ge d^{r(M)-a}$ and 
		$M$ is $d$-thick. 
	\end{lemma}
	\begin{proof}
		$\{\cl_M(X): X \in \cX\}$ is a cover of $\cX$ in $M$; since $\{E(M)\}$ is a $d$-minimal 
		cover of $\cX$, we have $\wt_M^d(\{\cl_M(X): X \in \cX\}) \ge \wt_M^d(\{E(M)\})$, so 
		$d^{a}\elem_M(\cX) \ge d^{r(M)}$, giving the first part of the lemma. 
		
		We will now show that $M$ is $d$-thick. Let $\cF$ be a cover of $M$ with flats of 
		smaller rank. If some $X \in \cX$ is not contained in $F$ for any set $F$ in $\cF$, 
		then $\{X \cap F: F \in \cF\}$ is a cover of $M|X$ of size at most $|\cF|$ with sets 
		of smaller rank than $X$, so $|\cF| \ge d$ by $d$-thickness of $X$. Otherwise, 
		$\cF$ is a $\cX$-cover, so $\wt_M^d(\cF) \ge \wt_M^d(\{E(M)\})$.  Therefore 
		$|\cF|d^{r(M)-1} \ge d^{r(M)}$, so $|\cF| \ge d$.  
	\end{proof}
	
	Our means of constructing scattered sets is the following lemma:. 
		
	\begin{lemma}\label{getscattered}
		Let $d \ge 1$ be an integer, $M$ be a matroid, and 
		$\cX \subseteq 2^{E(M)}$. If all sets in $\cX$ are $d$-thick in $M$, 
		and $\cF$ is a $d$-minimal cover of $\cX$ in $M$, then every subset 
		of $\cF$ is $d$-scattered in $M$. 
	\end{lemma}
	\begin{proof}
		Let $\cF' \subseteq \cF$. It is clear from $d$-minimality of $\cF$ that $\cF'$ is simple, 
		and that $\cF'$ is a $d$-minimal cover of $\cF'$. For each $F \in \cF'$, the set 
		$\{F\}$ is a $d$-minimal cover of $\{X \in \cX: X \subseteq F\}$ by $d$-minimality of 
		$\cF$, so by applying Lemma~\ref{maximalminimalcover} to $M|F$, we see that $F$ is 
		$d$-thick in $M$. Therefore $\cF'$ is $d$-scattered in $M$. 	
	\end{proof}
	
	In particular, if $\cF$ is a $d$-minimal cover of $M$ itself, then every subset of $\cF$ is 
	$d$-scattered in $M$, as the singleton $\{e\}$ is $d$-thick in $M$ for any $e \in E(M)$. 
					  
	\begin{lemma}\label{densityabsscattered}
		Let $a \ge 1$ and $d \ge 1$ be integers. If $M$ is a matroid, and $\cX \subseteq \cR_a(M)$ 
		is $d$-scattered in $M$, then $\elem_M(\cX) \le d^{r(M)-a}$. 
	\end{lemma}
	\begin{proof}
		 $\{E(M)\}$ is a cover of $\cX$ in $M$, so $d$-scatteredness of $\cX$ gives $d^{a}\elem_M(\cX) 
		 = \wt^d_M(\{\cl_M(X): X \in \cX\}) \le \wt^d_M(\{E(M)\}) = d^{r(M)}$, giving the result. 	
	\end{proof}
	
	The parameter $\tau^d$, for an appropriate $d$, is what we use to gain traction towards 
	Theorem~\ref{halfwaypoint}. Considering this parameter instead of $\tau_a$ is not a major 
	change in the setting of excluding $U_{a+1,b}$; indeed, these two parameters differ by at 
	most a constant factor. 
	
	\begin{lemma}\label{coveringcompare}
	If $a,b,d$ are integers with $1 \le a < b$ and $d \ge \tbinom{b}{a}$, and $M \in \fU(a,b)$, then no $d$-minimal cover of $M$ contains a set of rank greater than $a$, 
	and $\tau_a(M) \le \tau^d(M) \le d^{a}\tau_a(M)$. 
	\end{lemma}
	\begin{proof}
		Let $\cF$ be a $d$-minimal cover of $M$. By Lemma~\ref{getscattered}, every set in $\cF$ 
		is $d$-thick, so by Lemma~\ref{thickuniform} and definition of $d$, there is no set of 
		rank greater than $a$ in $\cF$. Therefore $\tau_a(M) \le |\cF| \le \wt^d_M(\cF) = \tau^d(M)$.
		Moreover, if $\cH$ is a minimum-sized cover of $M$ with sets of rank at most $a$, then
		$\tau^d(M) \le d^{a}|\cH| = d^{a}\tau_a(M)$. 
	\end{proof}	
	
	\section{Pyramids}\label{u1s3}

	We now define the intermediate structure that is vital to our proof. Let $a \ge 1$, $d \ge 1$, $q \ge 1$ 
	and $h \ge 0$ be integers, $M$ be a matroid, $\cS \subseteq \cR_a(M)$,  and 
	$\{e_1, \dotsc, e_h\}$ be an independent set of size $h$ in $M$.  For each $i \in \{0,1, \dotsc, h\}$, 
	let $M_i = M \dcon \{e_1, \dotsc, e_i\}$. 
	
	We say $(M, \cS; e_1, \dotsc, e_h)$ is an \textit{$(a,q,h,d)$-pyramid} if 
	\begin{itemize}
		\item $\cS \ne \varnothing$ and $S$ is skew to $\{e_1, \dotsc, e_h\}$ 
			for all $S \in \cS$, 
		\item for each $i \in \{0,1,\dotsc, h-1\}$ and $S \in \cS$, there are sets 
			$S_1, \dotsc, S_q \in \cS$, pairwise dissimilar in $M_i$ and each similar to 
			$S$ in $M_{i+1}$, and
		\item $S$ is $d$-thick in $M$ for all $S \in \cS$. 
	\end{itemize}	
	
	 A pyramid is a structured exponential-sized collection of thick rank-$a$ sets. 
	 For each $i \in \{0, \dotsc, h-1\}$ and each $S \in \cS$, contracting $e_{i+1}$ in $M_i$ `collapses'
	 the dissimilar $d$-thick sets $S_1, \dotsc, S_q$ onto the single $d$-thick set $S$ in $M_{i+1}$, 
	 without changing their rank.
	 
	 When $a = 1$, the set $\cS$ simply contains points; in this case, the value of $d$ is 
	 irrelevant and the structure described in the second condition is a set of $q$ other points
	 on a line through $e_{i+1}$. Pyramids are based on objects of the same 
	 name used by Geelen and Kabell in [\ref{gk}]; a pyramid in their sense is a special
	 sort of pyramid in our sense, with $a = 1$.
	
	
	The structure of a pyramid is self-similar, and the next two easily proved lemmas 
	concern smaller pyramids inside a pyramid:
	
	\begin{lemma}\label{shrinkpyramid}
		If $(M, \cS; e_1, \dotsc, e_h)$ is an $(a,q,h,d)$-pyramid, and $i$ and $j$ 
		are integers with $0 \le i \le j \le h$, then 
		\[(M \con \{e_{i+1}, \dotsc, e_j\}, \cS; e_1, \dotsc, e_i, e_{j+1}, \dotsc, e_h)\] 
		is an $(a,q,h-(j-i),d)$-pyramid.
	\end{lemma}
	
	\begin{lemma}\label{minorpyramid}
		Let $(M,\cS; e_1, \dotsc, e_h)$ be an $(a,q,h,d)$-pyramid and let $N$ be a minor of 
		$M \con \{e_1, \dotsc, e_h\}$. If $\cY \subseteq \cS \cap \cR_a(N)$, then there is a 
		minor $M'$ of $M$ and an $(a,q,h,d)$-pyramid $(M',\cS';e_1, \dotsc, e_h)$, so that 
		$\cY \subseteq \cS' \subseteq \cS$ and $N|\cY = (M' \con \{e_1, \dotsc, e_h\})|\cY$. 
	\end{lemma}
	
	The next lemma is our means of adding a 	`level' to a pyramid. In accordance with the 
	definition, it requires a point $e$ and a smaller pyramid on $M \con e$ such that $e$ 
	`lifts' each set in the pyramid into $q+1$ distinct sets. The proof, which we omit, is 
	cumbersome but routine. 
	
	\begin{lemma}\label{augmentpyramid}
		Let $M$ be a matroid, $e \in E(M)$ be a nonloop, $a,d,q,h$ be integers with 
		$q,a,d \ge 1$ and $h \ge 0$, and $\cX \subseteq \cR_a(M)$ be simple in $M$. 
		Let $\cX_{> q} = \{X \in \cX: |[X]_{M \con e} \cap \cX| > q\}$. If $M \con e$ 
		has an $(a,q+1,h,d)$-pyramid minor $P$ such that $\cS_P \subseteq \cX_{>q}$, then 
		$M$ has an $(a,q+1,h+1,d)$-pyramid minor $P'$ such that $\cS_{P'} \subseteq \cX$. 
	\end{lemma}
	
	The following lemma shows that a pyramid can be restricted to have bounded rank:
	
	\begin{lemma}\label{restrictpyramid}
		Let $(M, \cS; e_1, \dotsc, e_h)$ be an $(a,q,h,d)$-pyramid, let 
		$M_h = M \con \{e_1, \dotsc, e_h\}$, and let $S \in \cS$. There is a restriction $M'$ 
		of $M$ such that \[(M', \{S' \in \cS: S' \equiv_{M_h} S\}; e_1, \dotsc, e_h)\] is 
		an $(a,q,h,d)$-pyramid and $r(M') = a+h$.
	\end{lemma}
	\begin{proof}
		Let $M' = M|\cl_M(S \cup \{e_1, \dotsc, e_h\})$ and let 
		$\cS' = \{S' \in \cS: S' \equiv_{M_h} S\}$. Since $r_{M_h}(S) = a$, 
		we have $r(M') = a+h$. Let $0 \le i < h$. Let $S' \in \cS'$, and $S'_1, \dotsc, S'_q$ 
		be the sets for $i$ and $S'$ as given by the definition of a pyramid. Each $S'_j$ is 
		similar to $S'$ in $M_i$ and therefore also in $M_h$, so 
		$\{S'_1, \dotsc, S'_q\} \subseteq \cS'$ and 
		$(S'_1 \cup \dotsc \cup S'_q) \subseteq E(M')$. Therefore, 
		$(M', \cS'; e_1, \dotsc, e_h)$ is an $(a,q,h,d)$-pyramid. 
	\end{proof}
	
	Our penultimate lemma verifies that the set $\cS$ in a pyramid has exponential size:	
	\begin{lemma}\label{sizepyramid}
		If $(M, \cS; e_1, \dotsc, e_h)$ is an $(a,q,h,d)$-pyramid and $M_h = M \con \{e_1, \dotsc, e_h\}$, 
		then $\elem_M(\cS) \ge q^h\elem_{M_h}(\cS)$. 
	\end{lemma}
	\begin{proof}
		When $h = 0$, there is nothing to show. Otherwise, suppose that the result holds for 
		a fixed $h$, and let $(M, \cS; e_1, \dotsc, e_{h+1})$ be an $(a,q,h+1,d)$-pyramid. 
		We know that $(M \con e_1, \cS; e_2, \dotsc, e_{h+1})$ is an $(a,q,h,d)$-pyramid; 
		so $\elem_{M \con e_1}(\cS) \ge q^{h} \elem_{M_{h+1}}(\cS)$ by the 
		inductive hypothesis. Moreover, for each $S \in \cS$, there are pairwise dissimilar 
		sets $S_1, \dotsc, S_q \in \cS$, each similar to $S$ in $M \con e_1$. Therefore 
		$\elem_M(\cS) \ge q\elem_{M \con e_1}(\cS) \ge q^{h+1}\elem_{M_{h+1}}(\cS)$, 
		so the lemma holds. 		
	\end{proof}
	
	Finally, we observe that a pyramid has a restriction with bounded rank, containing an
	exponential-size subset of $\cS$. This lemma follows routinely from
	Lemmas~\ref{shrinkpyramid},~\ref{restrictpyramid} and~\ref{sizepyramid}. 
	
	\begin{lemma}\label{boundpyramid}
		If $(M, \cS; e_1, \dotsc, e_h)$ is an $(a,q,h,d)$-pyramid, and $h' \in \{0,1, \dotsc, h\}$ 
		is an integer,
		then there is a rank-$(a+h')$ restriction $M'$ of $M$ and a set $\cS' \subseteq \cS$, 
		so that $(M',\cS'; e_1, \dotsc, e_{h'})$ is an $(a,q,h',d)$-pyramid, and 
		$\elem_{M'}(\cS') \ge q^{h'}$. 
	\end{lemma}
	
	\section{Building a Pyramid}\label{u1s5}
	
	In this section, we show that a large $d$-scattered set allows us to either find a 
	$d$-firm subset of large rank in a minor, or a large pyramid. The majority of this 
	argument lies in an ugly technical lemma, which we will adapt into two useful corollaries. 
	To understand this lemma, it may be helpful to read it where $a_0 = 1$ and $a = 2$. In 
	this case, $\cX$ is a dense $d$-scattered set of points; the first outcome corresponds to a 
	$d$-point line minor whose points are in $\cX$, the second to a $(1,q+1,h,d)$-pyramid minor, 
	and the third to a minor containing a $d$-scattered collection of lines built from $\cX$. 
	
	\begin{lemma}\label{hardtechnical}
		There is an integer-valued function $f_{\ref{hardtechnical}}(a,d,h,m)$ so that, for 
		all integers $a_0,a,d,h,q$ with $a \ge a_0 \ge 1$, $d\ge 1$, $q \ge 1$, $h \ge 0$, 
		and $m \ge 0$, if $M$ is a matroid with $r(M) \ge f_{\ref{hardtechnical}}(a,d,h,m)$, 
		and a set $\cX \subseteq \cR_{a_0}(M)$ is $d$-scattered in $M$ and satisfies 
		$\elem_M(\cX) \ge r(M)^{f_{\ref{hardtechnical}}(a,d,h,m)} q^{r(M)}$, then either:
		\begin{enumerate}[(i)]
			\item\label{hti} 
				there is a minor $N$ of $M$ and a set $\cY \subseteq \cX \cap \cR_a(N)$ so that
				$r_N(\cY) > a$, and $\elem_N(\cY) \ge d^{r(N)-a}$, and 
				$\cl_N(\cY)$ is $d$-thick in $N$, or
			\item\label{htii} 
				$M$ has an $(a_0,q+1,h,d)$-pyramid minor $P$ with $\cS_P \subseteq \cX$, or 
			\item\label{htiii} 
				there exists an integer $a_1$ with $a_0 < a_1 \le a$, a minor $M'$ of $M$ 
				with $r(M') \ge m$, and a set $\cX' \subseteq \cR_{a_1}(M')$ so that 
				$\cX'$ is $d$-scattered in $M'$, and $\elem_{M'}(\cX') \ge r(M')^m q^{r(M')}$. 
		\end{enumerate}
	\end{lemma}
	\begin{proof}
		Let $a_0,a,d,h$ and $q$ be positive integers such that $a \ge a_0$, and let 
		$m \ge 0$ be an integer. Let $p_0 = 0$, and for each $h > 0$, recursively define 
		$p_h$ to be an integer so that 
		\[d^{-1}q^r(r-1)^{p_h-1}\left(p_h-3a(1+d^a)\right) \ge (r-1)^{p_{h-1}}q^{r-1},\]
		for all integers $r \ge 2$, and so that $p_h \ge \max(2,d,m+1)$.	
		
		We will show for all $h$ that if $M$ is a matroid with $r(M) \ge p_h$, and a 
		set $\cX \subseteq \cR_{a_0}(M)$ is $d$-scattered in $M$ and satisfies 
		$\elem_M(\cX) \ge r(M)^{p_h}q^{r(M)}$, then one of the three outcomes holds for $M$; 
		thus, setting  $f_{\ref{hardtechnical}}(a,d,h,m) = p_h$ will satisfy the lemma. 
		Our proof is by induction on $h$. If $h = 0$, then, since $(M,\{X\};)$ is an 
		$(a_0,q+1,0,d)$-pyramid for any $X \in \cX$, the outcome (\ref{htii}) holds. 
		Now fix $h > 0$ and suppose that the result holds for smaller $h$. 
		Let $p = p_h$, and $M$ be minor-minimal so that $r(M) \ge p$ and 
		there exists a $d$-scattered $\cX \subseteq \cR_{a_0}(M)$ such that 
		$\elem_{M'}(\cX) \ge r(M)^{p}q^{r(M)}$. Let $r = r(M)$. If $r = p$, 
		then $\elem_M(\cX) \ge p^pq^p > d^{p-a_0}$; this contradicts $d$-scatteredness 
		of $\cX$ by Lemma~\ref{densityabsscattered}, so we may assume that $r > p$. 
		
		By Lemma~\ref{pickcontract}, there is some $e \in E(M)$ so that 
		$\elem_M(\cX \cap \cR_{a_0}(M \con e)) \ge \left(1 - \tfrac{a_0}{r}\right)\elem_M(\cX)$. 
		Let $\cX' = \cX \cap \cR_{a_0}(M \con e)$ and $\cF$ be a $d$-minimal cover 
		of $\cX'$ in $M \con e$ such that $|\cF|$ is maximized. We may assume that all 
		sets in $\cF$ are flats of $M \con e$. The set $\cF$ is simple in $M \con e$; 
		for each $i \ge 1$, let $\cF_i = \cF \cap \cR_i(M \con e)$, noting that each $\cF_i$ is
		$d$-scattered in $M \con e$ by Lemma~\ref{getscattered}. We will henceforth 
		assume that~(\ref{hti}) and~(\ref{htiii}) do not hold. 
		
		\begin{claim}
			$\cF = \bigcup_{a_0 \le i \le a}\cF_i$.
		\end{claim}
		\begin{proof}[Proof of claim:]
			Every $F \in \cF$ must contain a set in $\cX'$, so $\cF$ contains no set of 
			rank less than $a_0$. If $\cF$ contains a set $F$ of rank greater than $a$, 
			then $\{F\}$ is a $d$-minimal cover of $\{X \in \cX': X \subseteq F\}$ in 
			$M \con e$, so by Lemma~\ref{maximalminimalcover}, the matroid $(M \con e)|F$ 
			and the set $\{X \in \cX': X \subseteq F\}$ satisfy (\ref{hti}), a contradiction.
		\end{proof}
		\begin{claim} 
			There is a set $\cX'' \subseteq \cX'$ that is $d$-scattered in $M \con e$ and 
			satisfies $\elem_M(\cX'') \ge q^r\left(r^p-a(1+d^a)r^{p-1}\right)$		
		\end{claim}
		\begin{proof}[Proof of claim:]			
			Each $X \in \cX'$ is contained in some set in $\cF$; for each $F \in \cF$, let 
			$\cX_F = \{X \in \cX': X \subseteq F\}$. By Lemma~\ref{densityabsscattered}, each 
			$F \in \cF$ satisfies $\elem_M(\cX_F) = \elem_{M|F}(\cX_F) \le d^{r(M|F)-a_0} \le
			 d^{a+1-a_0} \le d^a$. Moreover, each $\cF_i$ is simple and $d$-scattered in 
			 $M \con e$, so we may assume that $|\cF_i| \le r^mq^r$ for all $i > a_0$, as~(\ref{htiii}) does
			 not hold. Since $\cX'$ is the union of the $\cX_F$, we have 
			\begin{align*}
				\sum_{F \in \cF_{a_0}}(\elem_M(\cX_F)) &\ge \elem_{M}(\cX') -
				\sum_{\substack{a_0 < i \le a \\ F \in \cF_i}}\elem_M(\cX_F) \\
				&\ge (1 - \tfrac{a_0}{r})r^{p}q^{r} - d^a\sum_{a_0 < i \le a}|\cF_i| \\
				&\ge (1 - \tfrac{a}{r})r^pq^r - ad^ar^mq^r \\
				&\ge q^r(r^p - a(1+d^a)r^{p-1}),
			\end{align*}
			as $p-1 \ge m$. Let $\cX'' = \bigcup_{F \in \cF_{a_0}}\cX_F$. Now, since 
			$\cF_{a_0}$ is simple in $M \con e$, and every set in $\cF_{a_0}$ and every 
			set in $\cX'$ has rank $a_0$ in $M \con e$, no set in $\cX''$ is contained
			in two different sets in $\cF_{a_0}$. Therefore 
			$\elem_M(\cX'') = \sum_{F \in \cF_{a_0}}(\elem_M(\cX_F))$. Moreover, 
			$d$-minimality of $\cF$ implies that $\cF_{a_0} = \{\cl_{M \con e}(X): X \in \cX''\}$ 
			is a $d$-minimal cover of $\cX''$ in $M \con e$. Therefore $\cX''$ is 
			$d$-scattered in $M \con e$, giving the claim. 
		\end{proof}
		Let $\cY$ be a maximal subset of $\cX''$ that is simple in $M$. Since 
		$\cX''$ is $d$-scattered in both $M$ and $M \con e$, so is $\cY$. We have 
		$r(M \con e) = r-1 \ge p$, so minimality of $M$ gives 
		$|\elem_{M \con e}(\cY) < (r(M \con e))^pq^{r(M \con e)} = (r-1)^pq^{r-1}$. 
		Let $\cY_{> q} = \{Y \in \cY: |[Y]_{M \con e} \cap \cY| > q\}$ and 
		$\cY_{\le q} = \cY - \cY_{> q}$. Since $\cY$ is $d$-scattered and simple in $M$,
		Lemma~\ref{densityabsscattered} gives $|[Y]_{M \con e} \cap \cY| = 
		|\{Y' \in \cY: Y' \subseteq \cl_M(Y \cup \{e\})\}| \le d^{(a_0 +1)-a_0} = d$ 
		for all $Y \in \cY$. Now
		\begin{align*}
			q^r(r^p - a(1+d^a)r^{p-1}) &\le |\cY| \\
			&=|\cY_{> q}| + |\cY_{\le q}| \\
			&\le d \elem_{M \con e}(\cY_{>q}) + q \elem_{M \con e}(\cY_{\le q})\\
			&\le d \elem_{M \con e}(\cY_{>q}) + q \elem_{M \con e}(\cY) \\
			&< d \elem_{M \con e}(\cY_{>q})+ q(r-1)^pq^{r-1}.
		\end{align*}
		Rearranging this inequality yields 
		\begin{align*}
			\elem_{M \con e}(\cY_{>q}) &\ge d^{-1}q^r(r^p - (r-1)^p - a(1+d^a)r^{p-1}) \\
			&\ge d^{-1}q^r(p(r-1)^{p-1} - a(1+d^a)r^{p-1})\\
			&= d^{-1}q^r(r-1)^{p-1}\left(p - a(1+d^a)\left(\tfrac{r}{r-1}\right)^{p-1}\right).
		\end{align*}
		By hypothesis $r \ge p$, so $\left(\tfrac{r}{r-1}\right)^{p-1} \le 
		\left(\tfrac{p}{p-1}\right)^{p-1} \le 2.718\dotsc < 3$. This gives
		\begin{align*}
			\elem_{M \con e}(\cY_{>q}) &> d^{-1}q^r(r-1)^{p-1}\left(p - 3a(1+d^a)\right)\\	
			&\ge r(M \con e)^{p_{h-1}}q^{r(M \con e)}
		\end{align*}
		by definition of $p = p_h$. We may assume that (\ref{hti}) 
		and~(\ref{htiii}) both fail for $M \con e$ and $\cY_{>q}$; thus, by 
		induction on $h$, the matroid $M \con e$ has an $(a_0,q+1,h-1,d)$-pyramid 
		minor $P'$ with $\cS_{P'} \subseteq \cY_{> q}$. By Lemma~\ref{augmentpyramid}, 
		$M$ has an $(a_0,q+1,h,d)$-pyramid minor $P$ with 
		$\cS_P \subseteq \cY_{>q} \subseteq \cX$, which gives (\ref{htii}).
	\end{proof}
	
	Our first corollary, which will be used in the next section, finds a pyramid or a firm 
	set of rank greater than $a$, starting with a collection of thick rank-$a$ sets. The 
	corollary is obtained by specialising to the case where $a = a_0$, thus rendering the 
	third outcome impossible. 
	
	\begin{corollary}\label{buildpyramidcor}
		There is an integer-valued function $f_{\ref{buildpyramidcor}}(a,d,h)$ so that, 
		for any integers $a,d,h,q$ with $h \ge 0$, $a \ge 1$, $d \ge 2$ and $q \ge 1$, 
		if $M$ is a matroid such that $r(M) \ge f_{\ref{buildpyramidcor}}(a,d,h)$, 
		and $\cX \subseteq \cR_a(M)$ is a set such that every $X \in \cX$ is $d$-thick in $M$,
		and $\elem_M(\cX) \ge r(M)^{f_{\ref{buildpyramidcor}}(a,d,h)}q^{r(M)}$, then either
		\begin{enumerate}[(i)]
 			\item\label{bpc1} 
 				there is a minor $N$ of $M$, and a set $\cY \subseteq \cX \cap \cR_a(N)$ so that
 				$r_N(\cY) > a$ and $\cY$ is $d$-firm in $N$, or
 			\item\label{bpc2} 
 				$M$ has an $(a,q+1,h,d)$-pyramid minor $P$, with $\cS_P \subseteq \cX$. 
		\end{enumerate}
	\end{corollary}
	\begin{proof}
		Let $a,d,h,q$ be integers with $h \ge 0$, $a \ge 1$, $d \ge 2$ and $q \ge 1$. 
		Set $f_{\ref{buildpyramidcor}}(a,d,h) = f_{\ref{hardtechnical}}(a,d,h,0)$. Let 
		$M$ be a matroid such that $r(M) \ge f_{\ref{buildpyramidcor}}(a,d,h)$, and 
		$\cX \subseteq \cR_a(M)$ be a set such that every $x \in \cX$ is $d$-thick in $M$
		and $\elem_M(\cX) \ge r(M)^{f_{\ref{buildpyramidcor}}(a,d,h)}q^{r(M)}$. 
		We consider two cases: 
		
		\emph{Case 1:} $\cX$ is $d$-scattered in $M$.
		
		By definition of $f_{\ref{buildpyramidcor}}$, we can apply Lemma~\ref{hardtechnical} 
		to $\cX$. Since there is no integer $a_1$ with $a < a_1 \le a$, we know that
		~\ref{hardtechnical}(\ref{htiii}) cannot hold. If \ref{hardtechnical}(\ref{htii}) 
		holds, then we have our result. We may thus assume that \ref{hardtechnical}(\ref{hti}) 
		holds; now outcome~(\ref{bpc1}) follows from Lemma~\ref{firmdensity}. 

		\emph{Case 2:} $\cX$ is not $d$-scattered in $M$.
		
		By definition, $\{\cl_M(X): X \in \cX\}$ is not a $d$-minimal cover of $\cX$ in $M$, 
		so any $d$-minimal cover of $\cX$ contains a set $F$ of rank greater than $a$. Let 
		$\cX_F = \{X \in \cX: X \subseteq F\}$. The cover $\{F\}$ must be a $d$-minimal cover 
		of $\cX_F$, so by Lemma~\ref{maximalminimalcover} applied to $M|F$ and $\cX_F$, we
		have $\elem_M(\cX_F) \ge d^{r(M|F)-a}$. Again, outcome~(\ref{bpc1}) follows from
		Lemma~\ref{firmdensity}. 
	\end{proof}
	
	The second corollary essentially reduces Theorem~\ref{halfwaypoint} to the case where
	$M$ is a pyramid:
	
	\begin{corollary}\label{halfwayreduction}
		There is an integer-valued function $f_{\ref{halfwayreduction}}(a,b,d,h)$ so that, 
		for any integers $a,b,d,h,q$ with $q \ge 1$, $d \ge 2$, $h \ge 0$, and $1 \le a < b$, 
		if $M \in \fU(a,b)$ satisfies $r(M) > 1$ and 
		$\tau_a(M) \ge r(M)^{f_{\ref{halfwayreduction}}(a,b,d,h)}q^{r(M)}$,
		then there is some $a_0 \in \{1, \dotsc, a\}$ such that $M$ has an 
		$(a_0,q+1,h,d)$-pyramid minor. 
	\end{corollary}
	\begin{proof}
		Let $a,b,d,h,q$ be integers with $q \ge 1$, $d\ge 2$, $h \ge 0$ and $1 \le a < b$. 
		Let $d' = \max(d,\binom{b}{a})$. We define a sequence of integers $p_{a+1}, \dotsc, p_1$; 
		let $p_{a+1} = 0$, and for each $1 \le i \le a$, recursively set 
		$p_i = \max(p_{i+1},f_{\ref{hardtechnical}}(a,d',h,p_{i+1}))$. Note that 
		$p_1 \ge p_2 \ge \dotsc \ge p_{a+1}$. Set $f_{\ref{halfwayreduction}}(a,b,h,d)$ to 
		be an integer $p \ge p_1$ so that $a^{-1}(d')^{-a}r^{p} \ge r^{p_1}$
		for all integers $r \ge p_1$. Let $M$ be a matroid with $r(M) \ge p$ and 
		$\tau_a(M) \ge r(M)^pq^{r(M)}$.
		\begin{claim}
			Let $1 \le i \le a$. If $r(M) \ge p_i$, and $\cX \subseteq \cR_i(M)$ is $d'$-scattered in 
			$M$ and satisfies $\elem_M(\cX) \ge r(M)^{p_i}q^{r(M)}$, then $M$ has an $(a_0,q+1,h,d)$-pyramid 
			minor for some $i \le a_0 \le a$. 
		\end{claim}
		\begin{proof}[Proof of claim:]
			By definition of $p_i$, we can apply Lemma~\ref{hardtechnical} to $\cX$ in $M$. If 
			\ref{hardtechnical}(\ref{hti}) holds, then $M$ has a $d'$-thick minor of rank greater than $a$. 
			Since $d' \ge \binom{b}{a}$, this contradicts $M \in \fU(a,b)$ by Lemma~\ref{thickuniform}. 
			Since $d' \ge d$, \ref{hardtechnical}(\ref{htii}) gives the claim, so we may assume 
			that~\ref{hardtechnical}(\ref{htiii}) holds. If $i = a$, this is impossible, so the claim is proven. 
			Otherwise, we have the hypotheses for a minor of $M$ and some larger $i \le a$, so the 
			claim holds by induction. 
		\end{proof}
		Let $\cF$ be a $d'$-minimal cover of $M$. Clearly $\cF$ is simple. By Lemma~\ref{coveringcompare}, 
		we have $\wt^{d'}_M(\cF) \ge \tau_a(M)$ and every set in $\cF$ has rank at most $a$, so 
		$\elem_M(\cF) = |\cF| \ge (d')^{-a}\wt^{d'}_M(\cF) \ge (d')^{-a}r(M)^{p}q^{r(M)}$.  For each 
		$i \in \{1,\dotsc,a\}$, let $\cF_i = \cF \cap \cR_i(M)$. By a majority argument, some $i$ 
	 satisfies $\elem_M(\cF_i) = |\cF_i| \ge a^{-1}|\cF| \ge a^{-1}(d')^{-a}r(M)^pq^{r(M)} 
		\ge r(M)^{p_1}q^{r(M)} \ge r(M)^{p_i}q^{r(M)}$. The set $\cF_i$ is $d'$-scattered in $M$ by 
		Lemma~\ref{getscattered}, and $r(M) \ge p \ge p_i$, so the result follows from the claim. 
	\end{proof}

	\section{Finding Firmness}\label{u1s6}
	
	This section explores what can be done with a large collection $\cX$ of thick rank-$a$ sets in a 
	matroid $M$ with no large projective geometry as a minor. We prove a single lemma which finds a 
	large subcollection of $\cX$ that is firm in a minor of $M$. When $a = 1$ this is equivalent to 
	finding a large rank-$2$ uniform minor, and thus Theorems~\ref{gkpoly} and~\ref{gkexp} appear in 
	the base case of this lemma. 
		
	\begin{lemma}\label{easycase}
		There is an integer-valued function $f_{\ref{easycase}}(a,d,n,q)$ so that, for any positive 
		integers $a,d,n,q$, if $M$ is a matroid with $r(M) \ge f_{\ref{easycase}}(a,d,n,q)$, and 
		$\cX \subseteq \cR_a(M)$ is a set so that every $X \in \cX$ is $f_{\ref{easycase}}(a,d,n,q)$-thick 
		in $M$ and $\elem_M(\cX) \ge r(M)^{f_{\ref{easycase}}(a,d,n,q)}q^{r(M)}$, then either 
		\begin{enumerate}[(i)]
			\item\label{ec2} $M$ has a $\PG(n-1,q')$-minor for some $q' > q$, or
			\item\label{ec1} there is a minor $N$ of $M$ and a set $\cY \subseteq \cX \cap \cR_a(N)$ 
			so that $r_N(\cY) > a$ and $\cY$ is $d$-firm in $N$.
		\end{enumerate}
	\end{lemma}
	\begin{proof}
		Let $n,q,d$ be positive integers. Set 
		\[f_{\ref{easycase}}(1,d,n,q) = \max(2,f_{\ref{gkpoly}}(d,n),\lceil\alpha_{\ref{gkexp}}(d,n,q)\rceil).\] 
		We now define $f_{\ref{easycase}}(a,d,n,q)$ for general $a$ recursively; for each $a > 1$, suppose 
		that $t = f_{\ref{easycase}}(a-1,d,n,q)$ has been defined. Let $h$ be an integer so that $h \ge 3a + t$ and
		 $(3a)^{-1}d^{-3a}(q+1)^h \ge (h+a)^{t}q^{h+a}.$
		 Let $s = d^{h-a}$ and let $h'$ be an integer so that $h' \ge as + t$ and such that
		 $(as)^{-1}d^{-as}(q+1)^{h'} \ge (h'+a)^{t}q^{h'+a};$
		 Let $d' = \max(d,s+1)$ and set $f_{\ref{easycase}}(a,d,n,q) = \max(d',f_{\ref{buildpyramidcor}}(a,d',h+h')).$
		 
		 Let $a \ge 1$ be an integer, $M$ be a matroid with $r(M) \ge f_{\ref{easycase}}(a,d,n,q)$, 
		 and $\cX \subseteq \cR_a(M)$ be a set whose elements are all $f_{\ref{easycase}}(a,d,n,q)$-thick in 
		 $M$, satisfying $\elem_M(\cX) \ge r(M)^{f_{\ref{easycase}}(a,d,n,q)}$. We may assume that $M = M|\cX$; we show that $M$ satisfies 
		 (\ref{ec2}) or (\ref{ec1}), first resolving the case where $a = 1$, and proceeding by induction on $a$. 
		
		\begin{claim}If $a = 1$, then $M$ satisfies (\ref{ec2}) or (\ref{ec1}). 
		\end{claim}
		\begin{proof}[Proof of claim:]
		 Every $X \in \cX$ is a rank-$1$ set, and therefore 
		 $\tau_1(M)  \ge r(M)^{f_{\ref{easycase}}(1,d,n,q)}q^{r(M)}$. 
		 
		   If $q = 1$ then $r(M)^{f_{\ref{easycase}}(1,d,n,q)} \ge r(M)^{f_{\ref{gkpoly}}(d,n)}$, 
		   so if (\ref{ec2}) does not hold, then $M$ has a $U_{2,d}$-minor by Theorem~\ref{gkpoly}. 
		   This minor corresponds to a simple subset of $\cX$ in a rank-$2$ 
		   minor of $M$, containing $d$ pairwise dissimilar rank-$1$ sets. This is a rank-$2$, $d$-firm 
		   subset of $\cX$ in a minor of $M$, giving (\ref{ec1}). 
		   
		   If $q > 1$ then
		   $\tau_1(M) \ge f_{\ref{easycase}}(1,d,n,q)q^{r(M)} \ge \alpha_{\ref{gkexp}}(d,n,q)q^{r(M)},$
		   so the result follows from Theorem~\ref{gkexp} in a similar way to the $q = 1$ case. 
		\end{proof}
		
		Now, assume inductively that $a > 1$, and that $f_{\ref{easycase}}(\hat{a},\hat{d},n,q)$ as defined satisfies 
		the lemma for all $\hat{a} < a$, and all $\hat{d}$. Suppose further that (\ref{ec2}) 
		does not hold for $M$.
		
		\begin{claim}
			$M$ has an $(a,q+1,h+h',d')$-pyramid minor $P$ so that $\cS_P \subseteq \cX$.
		\end{claim}
		\begin{proof}[Proof of claim:]
			By definition of $f_{\ref{easycase}}(a,d,n,q)$, we know that 
			$r(M') \ge f_{\ref{buildpyramidcor}}(a,d,h+h')$, 
			$\elem_{M'}(\cX) \ge r(M)^{f_{\ref{buildpyramidcor}}(a,d',h+h')}q^{r(M)}$, and all sets in 
			$\cX$ are $d'$-thick in $M$; we can therefore apply Corollary~\ref{buildpyramidcor} to $M$. 
			Since $d' \ge d$, outcome~\ref{buildpyramidcor}(\ref{bpc1}) does not hold, 
			giving \ref{buildpyramidcor}(\ref{bpc2}) and hence the claim.
		\end{proof}
		Let $P = (M',\cS; e_1, \dotsc, e_{h+h'})$. By Lemma~\ref{boundpyramid}, we may assume that 
		$r(M') = h' + h+a$. Let $J = \{e_1, \dotsc, e_h\}$. By Lemma~\ref{shrinkpyramid}, 
		$(M' \con J, \cS; e_{h+1}, \dotsc, e_{h+h'})$ is an $(a,q+1,h',d)$-pyramid, so by Lemma~\ref{sizepyramid}, 
		there is a set $\cS' \subseteq \cS$ such that $|\cS'| \ge (q+1)^{h'}$ and $\cS'$ is simple in $M' \con J$. 	
		\begin{claim}
			There is a set $\cW \subseteq \cS'$ so that $|\cW| = s$ and $\cW$ is mutually skew in $M' \con J$. 		
		\end{claim}
		\begin{proof}[Proof of claim:]
			Suppose there is no such $\cW$.  By Lemma~\ref{findskew}, there is a minor $N$ of $M' \con J$ such that 
			$r(N) \ge r(M' \con J) - as$, a set $\cY \subseteq \cS' \cap \cR_a(N)$ such that 
			$|\cY| \ge (as)^{-1}|\cS'|$, and a nonloop $e$ of $N$ so that $e \in \cl_{N}(Y)$ for all $Y \in \cY$.
			We will apply the inductive hypothesis on $a$ to $N \con e$. 
			
			The set $\cY \subseteq \cS'$ is simple in $M \con J$, so by Lemma~\ref{firmdensityrel}, either (\ref{ec1}) holds or we have 
			\begin{align*}
			\elem_N(\cY) &\ge d^{r(N)-r(M' \con J)}\elem_{M' \con J}(\cY) \ge d^{-as}|\cY| \ge (as)^{-1}d^{-as}|\cS'| \\ 
			&\ge (as)^{-1}d^{-as}(q+1)^{h'}  \ge (h'+a)^{t}q^{h'+a}.
			\end{align*} Since 
			$r(N \con e) < r(M' \con J) = a+h'$, this gives 
			$\elem_{N}(\cY) \ge r(N \con e)^{t}q^{r(N \con e)}.$
			Let $\cY_0 = \{Y - \{e\}: Y \in \cY\}$. Because $e \in \cl_N(Y)$ for all $Y \in \cY$, we have $\cY_0 \subseteq \cR_{a-1}(N \con e)$, 
			and $\elem_{N \con e}(\cY_0) = \elem_{N}(\cY)$. Moreover, 
			$r(N \con e) \ge r(M' \con J)-as -1 \ge h'-as \ge t =f_{\ref{easycase}}(a-1,d,n,q)$, so by 
			the inductive hypothesis, there is a minor $N'$ of $N \con e$, and a set 
			$\cY_0' \subseteq \cY_0 \cap \cR_{a-1}(N')$ such that $r_{N'}(\cY') \ge a$, and 
			$\cY_0'$ is $d$-firm in $N'$. Let $\cY' = \{Y_0 \cup \{e\}: Y_0 \in \cY_0'\}$. If $N' = N \con (C \cup \{e\}) \del D$, where $e \notin C$ and
			$C \cup \{e\}$ is independent in $N$, then it is simple to check that 
			$\cY' \subseteq \cR_a(N \con C)$, that $r_{N \con C}(\cY') > a$, and that $\cY'$ is $d$-firm 
			in $N \con C$. This gives (\ref{ec1}). 			

		\end{proof}
		Let $\cW = \{W_1, \dotsc, W_s\}$, and for each $i \in \{1, \dotsc, s\}$, let 
		$\cS_i = \{S \in \cS: S \equiv_{M' \con J} W_i\}$. By Lemma~\ref{restrictpyramid} there is, 
		for each $i \in \{1, \dotsc, s\}$, a rank-$(a+h)$ restriction $M_i$ of $M'$ such that 
		$(M_i, \cS_i; e_1, \dotsc, e_h)$ is an $(a,q+1,h,d')$-pyramid. 
		
		\begin{claim}
			For each $i \in \{1, \dotsc, s\}$ there are distinct sets $V_i,Z_i,Z_i' \in \cS_i$ such that 
			$\{V_i,Z_i,Z_i'\}$ is mutually skew in $M_i$.
		\end{claim}
		\begin{proof}[Proof of claim:]
		By Lemma~\ref{sizepyramid}, $\cS_i$ has a subset $\cS'$ of size $(q+1)^h$ that is 
		simple in $M_i$. If there is a subset of $\cS'_i$ of size $3$ that is skew in $M_i$, 
		then the claim follows. Otherwise, by Lemma~\ref{findskew}, there is a minor $N_i$ of $M_i$, 
		with $r(N_i) \ge r(M_i)-3a$, a set $\cY \subseteq \cS'_i \cap \cR_a(N_i)$ such that 
		$|\cY| \ge (3a)^{-1}d^{-3a}|\cS'_i|$, and a nonloop $e$ of $N_i$ so that $e \in \cl_{N_i}(Y)$ 
		for all $Y \in \cY$. The proof is now very similar to that of the previous claim, 
		following from the definition of $h$.
		\end{proof}
		
		Let $\cV = \{V_1, \dotsc, V_s\}$. Since $V_i \equiv_{M \con J} W_i$ for each $i$, the set 
		$\cV$ is mutually skew in $M' \con J$. This last claim uses $Z_i$ and $Z_i'$ to contract the 
		elements of $\cV$, one by one, into the span of $J$ without reducing their rank, while 
		maintaining the `skewness' and structure of the elements of $\cV$ not yet contracted:
		
		\begin{claim}
			For each $i \in \{0, \dotsc, s\}$, there is a minor $N_i$ of $M$ such that 
		\begin{enumerate}[(a)]
				\item\label{ecca} $\{V_{i+1}, \dotsc, V_s\}$ is mutually skew in $N_i \con J$, 
				\item\label{eccb} $N_i |E(M_j) = M_j$ for each $j \in \{i+1, \dotsc, s\}$, and
				\item\label{eccc} $\{V_1, \dotsc, V_i\} \subseteq \cR_a(N_i|\cl_{N_i}(J))$, 
					and $\{V_1, \dotsc, V_i\}$ is simple in $N_i$.
		\end{enumerate}
		\end{claim}
		\begin{proof}[Proof of claim:]
			When $i = 0$, the claim is clear for $N_0 = M'$. Suppose inductively that $1 \le i \le s$
			and that the claim holds for smaller $i$. We will construct $N_i$ by contracting a rank-$a$ 
			set of $M_i = N_{i-1}|E(M_i)$, choosing one of its elements from $Z_i$ and the remaining $a - 1$ from $Z_i'$. By definition, $Z_i'$ and $V_i$ are similar to $W_i$ in $M_i \con J$, 
			so $r_{M_i \con J}(Z_i') = r_{M_i \con J}(V_i) = a$; Let $I \subseteq Z_{i}'$ be an independent 
			set of size $(a-1)$ in $M_i \con J$. So $\{V_i,Z_i\}$ is a skew pair of rank-$a$ sets in 
			$M_i \con I$, and $r(M_i \con I) = h+a-(a-1) = h+1$. Since $I$ is independent in $M_i \con J$, 
			it is skew to $J$ in $M_i$, so $r_{M_i \con I}(J) = h$. Moreover, 
			$r_{M_i \con (J \cup I)}(V_i) = r_{M_i \con (J \cup I)}(Z_i) = r_{M_i \con (J \cup I)}(Z_i') = 1$, 
			so neither $Z_i$ nor $V_i$ is contained in $\cl_{M_i \con I}(J)$. 
			
			By the inductive hypothesis, $(N_{i-1} \con I)|E(M_i) = M_i \con I$, so we can extend the 
			observations just made about $M_i \con I$ to apply in $N_{i-1} \con I$. Therefore, in the 
			matroid $N_{i-1} \con I$,  $\{V_i,Z_i\}$ is a skew pair of rank-$a$ sets, each contained 
			in the rank-$(h+1)$ set $E(M_i)$, which itself contains the rank-$h$ set $J$, and 
			$\cl_{N_{i-1} \con I}(J)$ does not contain $Z_i$ or $V_i$.

			For each $1 \le k < i$, let $F_k = V_i$ if $r_{N_{i-1} \con I}(V_k \cup V_i) > a+1$, 
			and $F_k = \cl_{N_{i-1} \con I}(V_k \cup V_i)$ otherwise. Since $V_i$ and $Z_i$ are skew sets 
			of rank $a > 1$ in $N_{i-1} \con I$, and $F_k$ is a flat of rank at most $a+1$ containing $V_i$, 
			it follows that $Z_i \not\subseteq F_k$, so $r_{N_{i-1} \con I}(F_k \cap Z_i) < a$. Also, 
			the set $\cl_{N_{i-1} \con I}(J)$ does not contain $Z_i$. The set $Z_i$ is $(d' \ge s+1)$-thick 
			in $N_{i-1} \con I$, and there are at most $s-1$ possible $k$, so there is some $f \in Z_i$ 
			that is not in any of the sets $F_k$, and not in $\cl_{N_{i-1} \con I}(J)$. Set 
			$N_i = N_{i-1} \con (I \cup \{f\})$. By choice of $f$, we have $r_{N_i}(J) = h = r_{N_{i-1}}(J)$, 
			so $I \cup \{f\}$ is skew to $J$ in $N_{i-1}$; we now show that $N_i$ 
			satisfies (\ref{ecca}), (\ref{eccb}) and (\ref{eccc}).  

			\begin{enumerate}[(a)]
				\item 
					We have $I \cup \{f\} \subseteq Z_i \cup Z_i'$. The sets $Z_i$ and $Z_i'$ are both 
					similar to $V_i$ in $M_i \con J = (N_{i-1}\con J)|E(M_i)$, so 
					$I \cup \{f\} \subseteq \cl_{N_{i-1} \con J}(V_i)$. The collection $\{V_i, \dotsc, V_s\}$ is mutually skew 
					in $N_{i-1} \con J$ by the inductive hypothesis, so $\{V_{i+1}, \dotsc, V_s\}$ is mutually skew 
					in $N_{i-1} \con (J \cup I \cup \{f\}) = N_i \con J$. 
				\item 
					Let $j \in \{i+1,\dotsc,s\}$. Since $(M_j, \cS_j; e_1, \dotsc, e_h)$ is an $(a,q+1,h,d)$-pyramid 
					and $V_j \in \cS_j$, the set $J \cup V_j$ is spanning in $M_j$ and $J$ is skew to $V_j$ 
					in $M_j$. As we saw in (\ref{ecca}), $I \cup \{f\}$ is skew to $J$ in $N_{i-1}$, 
					and is skew to $V_j$ in $N_{i-1} \con J$. Now, $M_j = N_{i-1}|E(M_j)$ and 
					$M_i = N_{i-1}|E(M_i)$, so 
					\begin{align*}
						r_{N_{i-1}}((I \cup \{f\}) \cup (J \cup V_j)) 
						&= r_{N_{i-1} \con J}(I \cup \{f\} \cup V_j) + r_{N_{i-1}}(J) \\
						&= r_{N_{i-1} \con J}(I \cup \{f\}) + r_{N_{i-1} \con J}(V_j) + r_{N_{i-1}}(J)\\
						&= r_{N_{i-1}}(I \cup \{f\}) + r_{N_{i-1}}(V_j) + r_{N_{i-1}}(J)\\
						&= r_{N_{i-1}}(I \cup \{f\}) + r_{N_{i-1}}(V_j \cup J). 
					\end{align*}
					Therefore, $I \cup \{f\}$ and $V_j \cup J$ are skew in $N_{i-1}$. Since $V_j \cup J$ 
					is spanning in $M_j$, this gives $N_i|E(M_j) = N_{i-1}|E(M_j) = M_j$. 
				\item 
					Since $I \cup \{f\}$ is skew to $J$ in $N_{i-1}$, it is clear that
					 $\{V_1, \dotsc, V_{i-1}\} \subseteq \cR_{a}(N_i|\cl_{N_i}(J))$ and that 
					 $\{V_1, \dotsc, V_{i-1}\}$ is simple in $N_i$. Moreover, $V_i$ is a rank-$a$ set 
					 that is skew to $Z_i \cup Z_i'$ in $N_{i-1}$, and therefore is skew to $I \cup \{f\}$, 
					 so $r_{N_i}(V_i) = a$. It therefore remains to show that $V_i$ is not similar in $N_i$ 
					 to any of $V_1, \dotsc, V_{i-1}$. 
		
					Suppose for a contradiction that $V_i \equiv_{N_i} V_k$ for some $k \in \{1,\dotsc,i\}$. Either 
					$V_i$ and $V_k$ are similar in $N_{i-1} \con I$, or $V_i$ and $V_k$ lie in a common 
					rank-$(a+1)$ flat $F$ of $N_{i-1} \con I$ and contracting $f \in F$ makes the two sets 
					similar in $N_i$. In the first case, the fact that $V_k$ and $V_i$ are skew in $N_{i-1}\con J$ gives 
					$0 = r_{N_{i-1} \con (I \cup  V_k)}(V_i) \ge r_{N_{i-1} \con (I \cup J)}(V_i) 
					= r_{N_{i-1} \con J}(V_i) - r_{N_{i-1}}(I) = a - (a-1) = 1$, a contradiction. 
					In the second case, we have $f \in \cl_{N_{i-1} \con I}(V_i \cup V_k)$, 
					which does not occur by choice of $f$.	
			\end{enumerate}
		\end{proof}
		
		Now, let $N = N_s|\cl_{N_s}(J)$. We have $r(N) \le h$, and $\cV$ is a simple subset of $\cR_a(N)$ 
		by construction, so $\elem_N(\cV) = |\cV| = s = d^{h-a}$. Also, $\cV \subseteq \cX$ and $d' \ge d$, 
		so every $V \in \cV$ is $d$-thick in $N$. (\ref{ec1}) now follows by applying 
		Lemma~\ref{firmdensity} to $\cV$ in $N$. 
	\end{proof}
		
	\section{Upgrading a Pyramid}\label{u1s7}
	
	The goal of this section is to prove that a sufficiently large pyramid minor will be enough to 
	prove Theorem~\ref{halfwaypoint}. We show that for very large $h$ and $d$, an $(a_0,q+1,h,d)$-pyramid 
	will either contain a thick set of rank greater than $a$, or a large projective geometry over $\GF(q')$ 
	for some $q' > q$. We first prove this when $a_0 = a$, and then show that for $a_0 < a$ we can find a 
	large pyramid as a minor with a larger $a_0$, thereby `upgrading' our pyramid.  
	
	An important intermediate object is an $(a_0,q+1,\cdot,\cdot)$-pyramid $P$ `on top of' a very firm set 
	$\cX \subseteq \cS_P$ with rank greater than $a_0$. We construct such objects using the results in the 
	previous section; this is the reason that we need to exclude a projective geometry. 
	
	We upgrade a pyramid of height $h$ on top of a firm set by `lifting' the firm set one level up the 
	pyramid $h$ times, sacrificing a large amount of firmness at each step. Our next two lemmas give
	the machinery needed for this; the first simply lifts a firm set up a pyramid of height $1$:
	
	\begin{lemma}\label{climbinductive}
		Let $a_0,a,q,d,d'$ be integers with $1 \le a_0 \le a$, $d,d' \ge 2$, and $q \ge 2$. 
		If $(M, \cS; e)$ is an $(a_0,q,1,d')$-pyramid, and $\cX \subseteq \cS$ is $d^{q+2}$-firm 
		in $M \con e$ and satisfies $r_{M \con e}(\cX) = a$, then either 
		\begin{enumerate}[(i)]
			\item\label{ci1} 
				there exists $\cY \subseteq \cS$ so that $r_M(\cY) = a+1$ and 
				$\cY$ is $d$-firm in $M$, or 
			\item\label{ci2} 
				there exist sets $\cX_1, \dotsc, \cX_q \subseteq \cS$ such that 
				\begin{itemize}
					 \item each $\cX_i$ is $d$-firm in $M$ and $r_M(\cX_i) = a$, and 
					 \item the $\cX_i$ are pairwise dissimilar in $M$, and each is 
					 		skew to $\{e\}$ in $M$, and similar to $\cX$ in $M \con e$.
				\end{itemize}
		\end{enumerate}
	\end{lemma}
	\begin{proof}
		We may assume that $\cX$ is spanning in $M \con e$, so $r(M) = a+1$. Suppose that the 
		first outcome does not hold. Let $I$ be an indexing set for $X$ (i.e. let 
		$\cX = \{X^i: i \in I\}$, with $|I| = |\cX|$). For each $i \in I$, let $X^i_1, \dotsc, X^i_q$ 
		be pairwise dissimilar sets in $\cS$, each similar to $X^i$ in $M$, as given by the 
		definition of a pyramid. 
		
		\begin{claim}
			There are sets $\cX_1, \dotsc, \cX_q \subseteq \cS$ and $I_1, \dotsc, I_q \subseteq I$ 
			such that the following conditions hold:
			\begin{itemize}
				\item for each $j \in \{1,\dotsc,q\}$, we have $\cX_j = \{X^i_j: i \in I_j\}$, 
				\item $I \supseteq I_1 \supseteq I_2 \supseteq \dotsc \supseteq I_q$, and
				\item each $j \in \{1,\dotsc,q\}$ satisfies $|\cX_j| \ge d^{-j}|\cX|$ and $r_M(\cX_j) \le a$. 
			\end{itemize}
		\end{claim}
		\begin{proof}[Proof of claim:]
			We construct the sets in question by induction on $j$. Suppose that $1 \le j < q$, and 
			that the sets $\cX_1, \dotsc, \cX_{j-1}$ and $I_1, \dotsc, I_{j-1}$ have been defined 
			to satisfy the conditions. Let $I_0 = I$, and $\cX_0 = \cX$; note that $|\cX_0| \ge d^0 |\cX|$.  
			As (\ref{ci1}) does not hold, the set $\{\cX^i_j: i \in I_{j-1}\}$ is not a rank-$(a+1)$, 
			$d$-firm set in $M$, so we may assume that there is some 
			$\cX_j  \subseteq \{X_i^j: I \in I_{j-1}\}$ such that $|\cX_j| \ge d^{-1}|\{\cX^i_j: i \in I_{j-1}\}|$ 
			and $r_M(\cX_j) \le a$. Now, 
			$|\cX_j| \ge d^{-1}|\{\cX^i_j: i \in I_{j-1}\}| = d^{-1}|I_{j-1}| = d^{-1}|\cX_{j-1}| \ge d^{-j}|\cX|$. 
			The set $\cX_j$, along with $I_j = \{i \in I_{j-1}: X^i_j \in \cX_j\}$, satisfies the required conditions. 
		\end{proof}
		\begin{claim}\label{cic1}
			For each $j \in \{1,\dotsc,q\}$, the set $\cX_j$ is $d$-firm in $M$ and $r_M(\cX_j) = r_{M \con e}(\cX_j) = a$. 
		\end{claim}
		\begin{proof}[Proof of claim]
			We know that $r_M(\cX_j) \le a$; let $\cX_j' \subseteq \cX_j$ satisfy $|\cX_j'| \ge d^{-1}|\cX_j|$, 
			and let $I_j' = \{i \in I_j: X_j^i \in \cX_j'\}$. Let $\cX' = \{X^i: i \in I_j'\}$. By definition 
			of $\cX$ and $\cX_i$, each set in $\cX'$ is similar in $M \con e$ to a set in $\cX'_j$, and vice versa. 
			We therefore have $|\cX'| = |\cX_j'|$ and $r_{M \con e}(\cX') = r_{M \con e}(\cX_j')$. 
			Now $|\cX'| = |\cX_j'| \ge d^{-1}|\cX_j| > d^{-(q+2)}|\cX|$, and $\cX' \subseteq \cX$, so 
			$d^{q+2}$-firmness of $\cX$ gives $r_{M \con e}(\cX') = r_{M \con e}(\cX) = a$. Therefore 
			\[ a \ge r_M(\cX_j) \ge r_M(\cX'_j) \ge r_{M \con e}(\cX'_j) = r_{M \con e}(\cX') = r_{M \con e}(\cX) = a,\]
			and the claim follows from definition of firmness and the fact that 
			$r_{M \con e}(\cX_j) \ge r_{M \con e}(\cX_j')$. 
		\end{proof}
	
		\begin{claim}
			The sets $\cX_j: j \in \{1, \dotsc, q\}$ are pairwise dissimilar in $M$. 
		\end{claim}
		\begin{proof}[Proof of claim:]
			Suppose not; let $\cX_j$ and $\cX_{j'}$ be similar in $M$, where $1 \le j < j' \le q$. 
			By~\ref{cic1}, $r_M(\cX_j \cup \cX_{j'}) = r_M(\cX_j) = a$. Let $i \in I_{j'}$. We have 
			$X^i_{j'}\in \cX_{j'}$ by definition, and $I_{j'} \subseteq I_j$, so $i \in I_j$ and 
			$X^i_j \in \cX_j$. But $X^i_j$ and $X^i_{j'}$ are dissimilar rank-$a_0$ sets in $M$, 
			each similar to the rank-$a_0$ set $X^i$ in $M \con e$. 
			Therefore, $e \in \cl_M(X^i_j \cup X^i_{j'})$, and so $e \in \cl_M(\cX_j \cup \cX_j') = \cl_M(\cX_j)$. 
			This contradicts the previous claim.
		\end{proof}
				
		By assumption, the set $\cX$ is spanning in the rank-$a$ matroid $M \con e$, and by the second 
		part of~\ref{cic1}, the set $\cX_j$ is also spanning in $M \con e$, so $\cX_j \equiv_{M \con e} \cX$. 
		By the claims above, (\ref{ci2}) follows. 
	\end{proof}
	
	The next lemma iterates the previous one $h$ times to upgrade a pyramid completely - here, 
	$a_0$ is upgraded to $a_1$ in the second outcome: 
	
	\begin{lemma}\label{pyramidbootstrap}
		Let $a_0,a_1,q$ and $d$ be integers with $1 \le a_0 \le a_1$ and $d,q \ge 2$, and let 
		$(M, \cS; e_1, \dotsc, e_h)$ be an $(a_0,q,h,d)$-pyramid. For each $0 \le i \le h$, let 
		$M_i = M \con \{e_1, \dotsc, e_i\}$. If $\cX \subseteq \cS$ is a set so that $r_{M_h}(\cX) = a_1$ 
		and $\cX$ is $d^{(q+2)^h}$-firm in $M_h$, then either
		\begin{enumerate}[(i)]
			\item\label{pbs1} there is an integer $i \in \{1, \dotsc, h\}$  and a set $\cY \subseteq \cS$ so that 
				$\cY$ is $d$-firm in $M_i$, and $r_{M_i}(\cY) > a_1$, or
			\item\label{pbs2} there is a set $\cT$ so that $(M, \cT; e_1, \dotsc, e_h)$ is an $(a_1,q,h,d)$-pyramid. 
		\end{enumerate}
	\end{lemma}
	\begin{proof}
		Assume that (\ref{pbs1}) does not hold; we will build a pyramid-like structure inductively. 
		\begin{claim}
			For each $i \in \{0,\dotsc,h\}$ there exists a nonempty collection $\fX^i$ of subsets of $\cS$ satisfying the following: 
			\begin{itemize}
				\item $\cl_M(\cX)$ is skew to $\{e_{i+1}, \dotsc, e_h\}$ in $M_i$ for all 
					$\cX \in \fX^i$, 
				\item for all $\cX \in \fX^i$ and $i'$ such that $i \le i' < h$, there exist sets 
					$\cX_1, \dotsc, \cX_q \in \fX^i$, pairwise dissimilar in $M_{i'}$ and each similar to 
					$\cX$ in $M_{i'+1}$, and 
				\item 
					every $\cX \in \fX^i$ satisfies $r_{M_i}(\cX) = a_1$ and is $d^{(q+1)^i}$-firm in $M_i$.
			\end{itemize}
		\end{claim}		
		\begin{proof}[Proof of claim:]
			Let $\fX^h = \{\cX\}$. It is clear that $\fX^h$ satisfies all three conditions. Fix $0 \le i < h$, 
			and suppose that $\fX^{i+1}$ has been defined to satisfy the conditions. Let $\cX \in \fX^{i+1}$. 
			We know that $(M_i,\cS; e_{i+1})$ is an $(a_0,q,1,d)$-pyramid; by the inductive hypothesis, the 
			set $\cX$ satisfies the hypotheses of Lemma~\ref{climbinductive} for this pyramid, and for $d^{(q+2)^i}$. 
			If \ref{climbinductive}(\ref{ci1}) holds, then so does outcome~(\ref{pbs1}) of the current lemma, 
			as $d^{(q+2)^i} \ge d$. Otherwise let $P(\cX) = \{\cX_1, \dotsc, \cX_q\}$, where 
			$\cX_1, \dotsc, \cX_q$ are the sets given by~\ref{climbinductive}(\ref{ci2}). Now 
			$\fX^i = \bigcup_{\cX \in \fX^{i+1}}P(\cX)$ will satisfy the claim, which follows inductively. 
		\end{proof}
		Let $\cT = \{\cl_M(\cX): \cX \in \fX^0\}$. Each set in $\fX^0$ is $d$-firm in $M$, so all sets in 
		$\cT$ are $d$-thick by Lemma~\ref{thicknessfirmness}. It is now clear from the claim that 
		$(M,\cT;e_1, \dotsc, e_h)$ is an $(a,q,h,d)$-pyramid.
	\end{proof}
	
	Having seen that a pyramid on top of a firm set is a useful object, we now show that such an object can be 
	constructed by Lemma~\ref{easycase} by excluding a projective geometry. 

	\begin{lemma}\label{pyramidfindfirm}
		There is an integer-valued function $f_{\ref{pyramidfindfirm}}(a_0,d,n,q,h)$ so that, for any 
		integers $a_0,d,n,q,d',h'$ with $a_0,d,n,q \ge 1$, $h \ge 0$, and 
		$\min(d',h') \ge f_{\ref{pyramidfindfirm}}(a_0,d,n,q,h)$, if $P$ is an $(a_0,q+1,h',d')$-pyramid 
		on a matroid $M$, then either
		\begin{enumerate}[(i)]
			\item\label{pff1}
				$M$ has a $\PG(n-1,q')$-minor for some $q' > q$, or
			\item\label{pff2}
				there is a minor $M'$ of $M$, an $(a_0,q+1,h,d)$-pyramid \[(M',\cS';e_1, \dotsc, e_h)\] 
				such that $\cS' \subseteq \cS_P$, and a set $\cY \subseteq \cS'$ such that $\cY$ is 
				$d$-firm in $M' \con \{e_1, \dotsc, e_h\}$ and $r_{M' \con \{e_1, \dotsc, e_h\}}(\cY) > a_0$. 
		\end{enumerate}
	\end{lemma}
	\begin{proof}
		Let $a_0,d,n,q$ be integers at least $1$. Let $h^*$  be an integer so that 
		$(q+1)^{h^*} \ge (a_0 + h^*)^{f_{\ref{easycase}}(a_0,d,n,q)}q^{a_0+h^*}$ and 
		$h^* \ge f_{\ref{easycase}}(a_0,d,n,q)$. Set $f_{\ref{pyramidfindfirm}}(a_0,d,n,q,h) = h+h^*$.  
		Now, let $h'$ and $d'$ be integers at least $h+h^*$, and $P = (M,\cS; e_1, \dotsc, e_{h'})$ 
		be an $(a_0,q+1,h',d')$-pyramid on a matroid $M$. We show that $M$ satisfies one of the two outcomes; 
		by Lemma~\ref{boundpyramid}, we may assume that $h' = h + h^*$ and that $r(M) = h+h^*+a_0$. 
		Let $M_h = M \con \{e_1, \dotsc, e_h\}$. 
		
		Now, $r(M_h) = h^* + a_0$, and $Q = (M_h, \cS; e_{h+1}, \dotsc, e_{h+h^*})$ is an 
		$(a_0,q+1,h^*,d')$-pyramid, and by Lemma~\ref{sizepyramid}, 
		$\elem_{M_h}(\cS) = (q+1)^{h^*} \ge (h^* + a_0)^{f_{\ref{easycase}}(a_0,d,n,q)}q^{h^*+a_0} 
		= r(M_h)^{f_{\ref{easycase}}(a_0,d,n,q)}q^{r(M_h)}$. Since 
		$d' \ge h \ge f_{\ref{easycase}}(a_0,d,n,q)$, we can apply Lemma~\ref{easycase} to $\cS$ in $M_h$. 
		We may assume that \ref{easycase}(\ref{ec2}) does not hold, so \ref{easycase}(\ref{ec1}) does; 
		therefore, there is a minor $N$ of $M_h$ and a set $\cY \subseteq \cS \cap \cR_a(N)$ such that 
		$r_N(\cY) > a_0$ and $\cY$ is $d$-firm in $N$. By Lemma~\ref{minorpyramid}, there is an 
		$(a_0,q+1,h,d')$-pyramid $(M', \cS'; e_1, \dotsc, e_h)$ so that $\cY \subseteq \cS'$ and 
		$N|\cY = (M' \con \{e_1, \dotsc, e_h\})|\cY$. Since $d' \ge d$, this gives (\ref{pff2}). 
	\end{proof}
	
	Finally, we combine the lemmas in this section to prove what we want: any $(a,q+1, h,d)$-pyramid 
	for very large $h$ and $d$ contains either a thick minor of rank greater than $a$, or a large projective 
	geometry over a field larger than $\GF(q)$. This tells us that finding such a pyramid is enough to 
	prove Theorem~\ref{halfwaypoint}. 
	
	\begin{lemma}\label{pyramidpromote}
		There is an integer-valued function $f_{\ref{pyramidpromote}}(a,d,n,q)$ so that, for any integers 
		$n,q,a_0,a,d,d^*,h^*$ with $n,q \ge 1$, $d \ge 2$, $1 \le a_0 \le a$, and 
		$\min(h^*,d^*) \ge f_{\ref{pyramidpromote}}(a,d,n,q)$, if $P$ is an $(a_0,q+1,h^*,d^*)$-pyramid 
		on a matroid $M$, then either
		\begin{enumerate}[(i)]
			\item\label{pp1} $M$ has a $\PG(n-1,q')$-minor for some $q' > q$, or
			\item\label{pp2} $M$ has a $d$-thick minor $N$ such that $r(N) > a$. 
		\end{enumerate}
	\end{lemma}
	\begin{proof}
		Let $n,q,a_0,a,d$ be integers with $n,q \ge 1$, $d \ge 2$, and $1 \le a_0 \le a$. For each pair of 
		integers $i,j$ with $1 \le i \le j \le a$, recursively define integers $h^i_j$ and $d^i_j$ as follows: 
		($h^i_j$ and $d^i_j$ are well-defined for all $i,j$ in the range, as $h^a_a$ and $d^a_a$ are defined, 
		and the definitions of $h^i_j$ and $d^i_j$ depend only on pairs $(i',j')$ exceeding $(i,j)$ lexicographically)  	
		\[h^i_j =  \begin{cases} f_{\ref{pyramidfindfirm}}(a,d,n,q,0) & \text{   if $j = a$} \\
								\max(h_{i+1}^i,f_{\ref{pyramidfindfirm}}(a,d^i_{i+1},n,q,h^i_{i+1})) & \text{   if $j < a$ and $i = j$}\\
								h_{i+1}^{i+1} + h^{i}_{j+1} &\text{   if $1 \le i < j < a$}
					\end{cases}\]
		\[d^i_j = \begin{cases}	\max(d,f_{\ref{pyramidfindfirm}}(a,d,n,q,0)) & \text{   if $j = a$} \\
								\max(d_{i+1}^i,f_{\ref{pyramidfindfirm}}(a,d^i_{i+1},n,q,h^i_{i+1})) & \text{   if $j < a$ and $i = j$}\\
								(\max(d^{i+1}_{i+1},d^{i}_{j+1}))^{(q+2)^{\left(h^{i+1}_{i+1}\right)}} &\text{   if $1 \le i < j < a$} 	
				  \end{cases}					
		\]
		Note that if $(i,j)$ exceeds $(i',j')$ lexicographically, then $h^i_j \le h^{i'}_{j'}$ and 
		$d \le d^i_j \le d^{i'}_{j'}$. We set $f_{\ref{pyramidpromote}}(a,d,n,q) = \max(h^1_1, d^1_1)$. The 
		lemma will follow from a technical claim:
		
		\begin{claim}
			Let $1 \le i \le j \le a$, and $d^{*} \ge d^i_j$ and $h^{*} \ge h^i_j$ be integers. 
			If $P = (M,\cS; e_1, \dotsc, e_{h^*})$ is an $(i,q+1,h^*,d^*)$-pyramid, and $\cX \subseteq \cS$ is 
			$d^*$-firm in $M \con \{e_1, \dotsc, e_{h}\}$ and satisfies $r_{M \con \{e_1, \dotsc, e_{h^*}\}}(\cX) = j$, 
			then (\ref{pp1}) or (\ref{pp2}) holds for $M$. 
		\end{claim}
		\begin{proof}[Proof of claim:]
			By Lemma~\ref{boundpyramid}, we may assume that $h^* = h^i_j$. If $j = a$, then 
			$h^i_j$ and $d^i_j$ are both at least $f_{\ref{pyramidfindfirm}}(a,d,n,q,0)$; we can therefore apply Lemma~\ref{pyramidfindfirm} 
			to $P$. Outcome \ref{pyramidfindfirm}(\ref{pff1}) gives (\ref{pp1}), and applying 
			Lemma~\ref{thicknessfirmness} to the $\cX$ and $M'$ given by \ref{pyramidfindfirm}(\ref{pff2}) 
			gives (\ref{pp2}). Suppose inductively that $1 \le i \le j < a$, and that the claim holds for 
			all $(i',j')$ lexicographically greater than $(i,j)$.
			
			If $j = i$, then let $h' = h_{i+1}^i$. By Lemma~\ref{pyramidfindfirm}, there is a minor $M'$ of $M$, an 
			$(i,q+1,h',d^i_{i+1})$-pyramid $(M',\cS'; e_1, \dotsc, e_{h'})$, and a set 
			$\cX' \subseteq \cS'$ so that $\cX'$ is $d^i_{i+1}$-firm in $M' \con \{e_1, \dotsc, e_{h'}\}$ 
			and $r_{M' \con \{e_1, \dotsc, e_{h'}\}}(\cX') \ge i+1$. Let $i' = r_{M'}(\cX)$. If $i' > a$, then outcome (\ref{pp2}) holds by Lemma~\ref{thicknessfirmness}. Otherwise, since $h' = h^i_{i+1} \ge h^i_{i'}$ and $d^i_{i+1} \ge d^i_{i'}$, 
			the lemma follows from the inductive hypothesis. 
			
			We may now assume that $1 \le i < j < a$. For each $k \in \{0,\dotsc,h\}$, write $M_k$ for
			$M \con \{e_1, \dotsc, e_k\}$.  Now $h^* = h^i_j = h^i_{j+1} + h^{i+1}_{i+1}$; let 
			$h' = h^i_{j+1}$ and $h'' = h^{i+1}_{i+1}$. By Lemma~\ref{shrinkpyramid}, 
			$P' = (M_{h'}, \cS; e_{h'+1}, \dotsc, e_{h' + h''})$ is an $(i,q+1,h'',d^i_j)$-pyramid, 
			and $\cX$ is $d^i_j$-firm in $M_h = M_{h'} \con \{e_{h'+1}, \dotsc, e_h\}$. Since
			$d^* \ge (\max(d^{i+1}_{i+1},d^i_{j+1}))^{(q+2)^{h''}}$, we can apply Lemma~\ref{pyramidbootstrap} to $P'$. 
 
 			If \ref{pyramidbootstrap}(\ref{pbs1}) holds for $P'$, then there is an $\ell \in \{1, \dotsc, h''\}$ and a set 
 			$\cY \subseteq \cS$ that is $d^i_{j+1}$-firm in 
 			$M_{h'} \con \{e_{h'+1}, \dotsc, e_{h'+\ell}\} = M_{h' + \ell}$ and satisfies $r_{M_{h'+\ell}}(\cY) > j$; 
 			let $j ' = r_{M_{h'+\ell}}(\cY)$. If $j' > a$, then Lemma~\ref{thicknessfirmness} gives outcome (\ref{pp2}). Otherwise 
 			$(M \con \{e_{h'+1}, \dotsc, e_{h'+\ell}\}, \cS; e_1, \dotsc, e_{h'})$ is an $(i,q+1,h',d)$-pyramid by Lemma~\ref{shrinkpyramid}, 
 			and since $d^* \ge d^i_{j+1} \ge d^i_{j'}$ and $h' = h^i_{j+1} \ge h^{i}_{j'}$, this pyramid and the 
 			set $\cY$ satisfy the hypotheses of the claim for $(i,j')$. The claim follows by induction. 
			
			If \ref{pyramidbootstrap}(\ref{pbs2}) holds for $P$, then there is a $(j,q+1,h'',d^{i+1}_{i+1})$-pyramid 
			$Q$ on $M_{h'}$. We have $h'' = h^{i+1}_{i+1} \ge h^{j}_{j}$, and for any $X \in \cS_Q$ the set $\{X\}$ 
			is trivially $d^{j}_{j}$-firm in $M_{h'}$, so $Q$ and $\{X\}$ satisfy the hypotheses of the claim for 
			$(j+1,j+1)$. Again, the claim follows inductively. 
			
		\end{proof}
		Let $h^*$ and $d^*$ be integers with $\min(h^*,d^*) \ge f_{\ref{pyramidpromote}}(a,d,n,q)$, and 
		$P = (M, \cS; e_1, \dotsc, e_{h^*})$ be an $(a_0,q+1,h^*,d^*)$-pyramid. For any $X \in \cS_P$, the 
		set $\{X\}$ is $d^*$-firm in $M \con \{e_1, \dotsc, e_{h^*}\}$, and 
		$d^* \ge f_{\ref{pyramidpromote}}(a,d,n,q) \ge d^1_1 \ge d^{a_0}_{a_0}$. Moreover, 
		$h^* \ge f_{\ref{pyramidpromote}}(a,d,n,q) \ge h^1_1 \ge h^{a_0}_{a_0}$, so the lemma follows 
		by applying the claim to $P$ and $\{X\}$. 
	\end{proof}
	
	\section{The Main Theorems}

	We are now able to prove Theorem~\ref{halfwaypoint}, which we restate here for convenience:
	 
	\begin{theorem}\label{halfwaypointrep}
		There is an integer-valued function $f_{\ref{halfwaypointrep}}(a,b,n,q)$ so that, for any integers 
		$1 \le a < b$, $q \ge 1$ and $n \ge 1$, if $M \in \cU(a,b)$ satisfies $r(M) > 1$ and  
		$\tau_a(M) \ge r(M)^{f_{\ref{halfwaypointrep}}(a,b,n,q)}q^{r(M)}$, then $M$ has a $\PG(n-1,q')$-minor 
		for some prime power $q' > q$.
	\end{theorem}
	\begin{proof}
		Let $a,b,n,q$ be integers with $n,q \ge 1$ and $1 \le a < b$. Let $d = \binom{b}{a}$, and 
		$h = f_{\ref{pyramidpromote}}(a,d,n,q)$.  Set $f_{\ref{halfwaypointrep}}(a,b,n,q)$ to be an integer 
		$p$ such that $p \ge f_{\ref{halfwayreduction}}(a,b,h,h)$, and so that $r^p \ge d^r$ for all $r$ 
		such that $2 \le r < p$. 
		
		Let $M \in \cU(a,b)$ satisfy $r(M) > 1$ and $\tau_a(M) \ge r(M)^pq^{r(M)}$; we show that 
		$M$ has a $\PG(n-1,q')$-minor for some $q' > q$.  If $r(M) < p$, then by Theorem~\ref{kdensity}, 
		$\tau_a(M) \le \binom{b-1}{a}^{r(M)} < d^{r(M)} \le r(M)^p$, a contradiction. So we may assume that 
		$r(M) \ge p$. By Lemma~\ref{halfwayreduction}, $M$ has an $(a_0,q+1,h,h)$-pyramid minor for some 
		$1 \le a_0 \le a$. By Lemma~\ref{pyramidpromote}, $M$ either has a $\PG(n-1,q')$-minor for some 
		$q' > q$, giving the theorem, or a $d$-thick minor of rank greater than $a$, in which case a 
		contradiction follows from Lemma~\ref{thickuniform}. 
	\end{proof}
	
	We now derive Theorem~\ref{mainnice}, which we also restate, as a consequence:
	
	\begin{theorem}\label{mainnicerep}
		Let $a \ge 1$ be an integer. If $\cM$ is a minor-closed class of matroids, then either
		\begin{enumerate}
			\item\label{mri} $\tau_a(M) \le r(M)^{n_{\cM}}$ for all $M \in \cM$, or
			\item\label{mrii} there is a prime power $q$ so that $\tau_a(M) \le r(M)^{n_{\cM}}q^{r(M)}$ 
				for all $M \in \cM$, and $\cM$ contains all $\GF(q)$-representable matroids, or
			\item\label{mriii} $\cM$ contains all rank-$(a+1)$ uniform matroids. 
		\end{enumerate}
	\end{theorem}
	\begin{proof}
		We may assume that (\ref{mriii}) does not hold, so there is some $b$ such that $\cM \subseteq \cU(a,b)$. 
		Moreover, the uniform matroid $U_{a+1,b}$ is simple and $\GF(q)$-representable for all 
		$q \ge b$ (see [\ref{hirschfeld}]), so $\PG(a,q) \notin \cM$ for all $q \ge b$. Therefore, 
		there is some $q_0 < b$ and some $n_0 > a$ such that $\PG(n_0-1,q) \notin \cM$ for all 
		$q > q_0$; choose $q_0$ to be minimal such that $q_0$ is either $1$ or a prime power, 
		and this $n_0$ exists. 
		
		By choice of $q_0$, we have $\tau_a(M) \le r(M)^{f_{\ref{halfwaypointrep}}(a,b,n_0,q_0)}q_0^{r(M)}$ 
		for all $M \in \cM$ by Theorem~\ref{halfwaypointrep}. If $q_0 = 1$ then this gives (\ref{mri}), 
		and if $q_0$ is a prime power greater than $1$, then minimality of $q_0$ implies that 
		$\PG(n-1,q_0) \in \cM$ for all $n \ge 1$, giving (\ref{mrii}). 
	\end{proof}

	\section*{References}
\newcounter{refs}

		\begin{list}{[\arabic{refs}]}
{\usecounter{refs}\setlength{\leftmargin}{10mm}\setlength{\itemsep}{0mm}}

\item\label{openprobs}
J. Geelen, 
Some open problems on excluding a uniform matroid, 
Adv. in Appl. Math. 41(4) (2008), 628--637.

\item\label{gk}
J. Geelen, K. Kabell,
Projective geometries in dense matroids, 
J. Combin. Theory Ser. B 99 (2009), 1--8.

\item\label{gkb}
J. Geelen, K. Kabell, 
The {E}rd{\H o}s-{P}\'osa property for matroid circuits,
J. Combin. Theory Ser. B 99 (2009), 407--419.  

\item\label{gkw}
J. Geelen, J.P.S. Kung, G. Whittle,
Growth rates of minor-closed classes of matroids,
J. Combin. Theory. Ser. B 99 (2009), 420--427.

\item\label{gw}
J. Geelen, G. Whittle,
Cliques in dense $\GF(q)$-representable matroids, 
J. Combin. Theory. Ser. B 87 (2003), 264--269.

\item\label{hirschfeld}
J. W. P. Hirschfeld,
Complete Arcs, 
Discrete Math. 174(1-3):177--184 (1997),
Combinatorics (Rome and Montesilvano, 1994).

\item\label{part2}
P. Nelson, 
Projective geometries in exponentially dense matroids. II, 
In preparation.

\item\label{thesis}
P. Nelson,
Exponentially Dense Matroids,
Ph.D thesis, University of Waterloo (2011). 

\item \label{oxley}
J. G. Oxley, 
Matroid Theory,
Oxford University Press, New York (2011).
\end{list}		
			\end{document}